\documentclass[]{amsart}

\usepackage{fullpage}
\usepackage{amsmath,amssymb,amsthm}
\usepackage{mathrsfs}
\usepackage{hyperref}
\usepackage{tikz}
\usetikzlibrary{matrix,arrows}

\newtheorem{theorem}{Theorem}[section]
\newtheorem{corollary}[theorem]{Corollary}
\newtheorem{proposition}[theorem]{Proposition}
\newtheorem{definition}[theorem]{Definition}
\newtheorem{lemma}[theorem]{Lemma}

\theoremstyle{definition}

\newtheorem{remark}[theorem]{Remark}

\newcommand{\twopartdef}[4]
{
	\left\{
		\begin{array}{ll}
			#1 & \mbox{if } #2 \\
			#3 & \mbox{if } #4
		\end{array}
	\right.
}

\newcommand{\hatL}{\mathcal{L}}
\newcommand{\C}{\mathcal{C}}
\newcommand{\curvedLie}{\hat{\mathscr{L}}}
\newcommand{\MC}{\operatorname{MC}}
\newcommand{\Hom}{\operatorname{Hom}}
\newcommand{\ho}{\operatorname{ho}}
\newcommand{\gr}{\operatorname{gr}}
\newcommand{\id}{\operatorname{id}}
\newcommand{\finL}{\mathrm{f}\mathbb{Q}\mathrm{-ho}\curvedLie_*}
\newcommand{\finA}{\mathrm{f}\mathbb{Q}\mathrm{-ho}\mathscr{A}}
\newcommand{\finS}{\mathrm{fN}\mathbb{Q}\mathrm{-ho}\mathscr{S}}

\title{Unbased rational homotopy theory: a Lie algebra approach}
\author{James Maunder}
\address{Max-Planck-Institut f\"ur Mathematik\\
Vivatsgasse 7\\
53111 Bonn\\
Germany}
\email{maunder@mpim-bonn.mpg.de}

\hypersetup{	pdfmenubar=true}

\begin{document}

\begin{abstract}
In this paper an algebraic model for unbased rational homotopy theory from the perspective of curved Lie algebras is constructed. As part of this construction a model structure for the category of pseudo-compact curved Lie algebras with curved morphisms will be introduced; one which is Quillen equivalent to a certain model structure of unital commutative differential graded algebras, thus extending the known Quillen equivalence of augmented algebras and differential graded Lie algebras.
\end{abstract}

\maketitle

\tableofcontents

\section*{Introduction}

Rational homotopy theory of connected spaces was developed by Quillen \cite{quillen} from the viewpoint of differential graded Lie algebras and by Sullivan \cite{sullivan} from the viewpoint of commutative differential graded algebras. A standard reference for the correspondence between rational connected spaces---in both the pointed and unpointed cases---and commutative differential graded algebras is \cite{bousfield_gugenheim}. Section \ref{sec_rat_homo} relies heavily upon results of that paper. Within \cite{bousfield_gugenheim} a closed model category structure is constructed for the category of non-negatively graded unital commutative differential graded algebras. Further, it is shown that there exists a pair of Quillen adjoint functors between the category of non-negatively graded commutative differential graded algebras and the category of simplicial sets. These Quillen adjoint functors restrict to an equivalence on the homotopy categories of connected commutative differential graded algebras that admit a minimal model of finite type over $\mathbb{Q}$, and connected rational nilpotent simplicial sets of finite type over $\mathbb{Q}$. In the work of Lazarev and Markl \cite{laz_markl}, this correspondence between non-negatively graded unital commutative differential graded algebras and simplicial sets is generalised to create a disconnected rational homotopy theory by extending to the category of $\mathbb{Z}$-graded unital commutative differential graded algebras. Removing the restriction of non-negative grading, despite seeming relatively harmless, has profound consequences; for example, a commutative algebra concentrated in degree zero is necessarily cofibrant in the category of Bousfield and Gugenheim, but this is not always the case in the category of Lazarev and Markl. The latter category does, however, appear to be more natural as one can, for example, define Harrison-Andr\'e-Quillen cohomology within it, cf.~\cite{block_laz}. The theory of \cite{laz_markl} relies upon a Quillen adjoint pair of functors between the categories of $\mathbb{Z}$-graded commutative differential graded algebras and simplicial sets. This functor gives rise to an adjunction on the level of homotopy categories and, moreover, restricts to an equivalence between the homotopy category of simplicial sets having a finite number of connected components, each being rational, nilpotent, and of finite type (over $\mathbb{Q}$) and a certain subcategory of the homotopy category of $\mathbb{Z}$-graded commutative differential graded algebras (given explicitly loc.~cit.). Sections \ref{sec_modelCat_main} and \ref{sec_rat_homo} rely on some of the work developed by Lazarev and Markl in their paper. Further, Lazarev and Markl constructed a second version of this disconnected rational homotopy theory: one formed using differential graded Lie algebras. This second version was performed by relating the homotopy theory of unital commutative differential graded algebras and differential graded Lie algebras. Relationships of this kind are often referred to as a Koszul duality, the theory of which was established by work of Quillen \cite{quillen}. Quillen showed there existed a duality between differential graded counital cocommutative coalgebras and differential graded Lie algebras under (quite severe) restrictions on the grading of the objects considered. Subsequently, these restrictions were removed by Hinich, cf.~\cite{hinich_stacks}.

The context of this paper is one in which attention is drawn to the categories of commutative differential graded algebras and pseudo-compact curved Lie algebras. A curved Lie algebra is said to be pseudo-compact if it is the inverse limit of an inverse system of finite dimensional nilpotent curved Lie algebras. This inverse limit is taken using strict morphisms, i.e.~using morphisms that commute with the differentials and curvature elements. Herein it will be shown that there exists a Quillen equivalence between the categories of commutative differential graded algebras and pseudo-compact curved Lie algebras. This result is proven by a suitable adaptation of Hinich's methods \cite{hinich_stacks} and influenced by the work of Positselski \cite{positselski}. Consequently an algebraic model of unbased disconnected rational homotopy theory using pseudo-compact curved Lie algebras is obtained. The condition that the curved Lie algebras be pseudo-compact is necessary, for it cannot be removed even when restricted to connected spaces (or simplicial sets).

Numerous papers now discuss the homotopy theory of differential graded coalgebras over different operads. For example Positselski studied coassociative differential graded coalgebras \cite{positselski}. However, the coalgebras were assumed to be conilpotent loc.~cit., and thus the homotopy theory developed therein is not completely general. Furthermore, Positselski worked with curved objects suggesting that when discussing a Koszul duality in more general cases  one side of the Quillen equivalence should be a category consisting of curved objects; a hypothesis that is strengthened by results of \cite{chuang_laz_mannan} and this paper.

Section \ref{sec_curvedLie} begins by defining and discussing some of the basic properties of curved Lie algebras and their morphisms, before moving on to discuss the category of pseudo-compact curved Lie algebras and filtrations of such objects. These filtrations are particularly important to this paper and the homotopy theory developed herein; the filtrations play an essential role in defining the correct notion of a weak equivalence in the category of pseudo-compact curved Lie algebras, cf.~Section \ref{sec_modelCat_main}. Without these filtrations the usual notion of a weak-equivalence, i.e.~a quasi-isomorphism, is simply not fine enough.

Section \ref{sec_adjoint} introduces a pair of adjoint functors between the categories of pseudo-compact curved Lie algebras and unital commutative differential graded algebras; these functors are analogues of the Chevalley-Eilenberg and Harrison complexes (see \cite{weibel} and \cite{barr,harrison} respectively) in homological algebra and are influenced by constructions of Positselski in the associative case, see \cite{positselski}.

A model structure for the category of pseudo-compact curved Lie algebras is defined in Section \ref{sec_modelCat_main}; one which, by the previously mentioned pair of adjoint functors, is Quillen equivalent to the existing model structure of unital commutative differential graded algebras given by Hinich \cite{hinich_homo_alg}---the main result of the paper. As remarked above, the notion of a filtered quasi-isomorphism using the filtrations defined in Section \ref{sec_curvedLie} plays a fundamental role in the homotopy theory of pseudo-compact curved Lie algebras. The proof of this Quillen equivalence relies upon the duality for associative algebras contained within \cite{positselski}, as the functors defined in Section \ref{sec_adjoint} can be `embedded' into the adjunction given in op.~cit., this is similar to the method used in \cite{laz_markl}.

In Section \ref{sec_rat_homo} this equivalence of homotopy categories is then applied, in a similar manner to Lazarev and Markl \cite{laz_markl}, in the construction of a disconnected rational homotopy theory from the perspective of pseudo-compact curved Lie algebras. This viewpoint results in a couple of corollaries. The first corollary shows that the Maurer-Cartan simplicial set functor commutes with coproducts when restricted to the correct subcategories of curved Lie algebras and simplicial sets; this result is an analogue of one proven in \cite{laz_markl}, but the proof herein is much more simple due to the material developed. The second corollary constructs Lie models for mapping spaces between two simplicial sets each composed of finitely many connected components where each component is rational, nilpotent, and of finite type. More precisely, it constructs the mapping space as the Maurer-Cartan simplicial set of a certain curved Lie algebra. This approach to calculating the mapping space is building upon previous results that used rational homotopy theory \cite{BPS,BES,haefliger} and using Lie models \cite{berglund,BFM,BFM2}. The approach via the Maurer-Cartan simplicial set contained herein, however, lends itself to much more calculations than the previous results and is much like work of Lazarev \cite{andrey_MC}. In fact, under suitable conditions, one can combine the results \cite[Theorem 8.1]{andrey_MC} and Corollary \ref{cor_mapping_space} to construct a model for the mapping space as a disjoint union of Maurer-Cartan simplicial sets.

\section*{Notation and conventions}

Throughout this paper it is assumed that all commutative and Lie algebras are over a fixed base field, $k$, of characteristic $0$. There are no further assumptions made about the base field until Section \ref{sec_rat_homo}, when the field will necessarily be assumed to be $\mathbb{Q}$, the rational numbers. All unmarked tensors will be over the base field. Every graded space will be assumed to be $\mathbb{Z}$-graded, unless stated otherwise: Section \ref{sec_rat_homo}, for example, will assume some algebras to be $\mathbb{N}_0$-graded. All commutative algebras possess a unit, unless stated otherwise, and likewise that all cocommutative coalgebras possess a counit.

The notation $\langle a, b, c \dots \rangle$ will be understood to mean the graded $k$-vector space spanned by the basis vectors $a,b,c,\dots $, where the degrees of $a,b,c,\dots$ are specified. Given a graded vector space $V$, the notation $\hat{L}V$ will be understood to mean the completed free graded Lie algebra on the generators of $V$. More specifically,  $\hat{L}V$ is the subspace of the completed tensor algebra $\prod_{i\geq 1} V^{\otimes i}$ spanned by Lie monomials.

This paper will often use an assortment of abbreviations, including: `dg' for `differential graded'; `dgla' for `differential graded Lie algebra'; `cdga' for `commutative differential graded algebra'; `CMC' for `closed model category' (in the sense of \cite{quillen}, for a review of this material consult \cite{dwyer_spalinski}); `LLP' for `left lifting property'; and `RLP' for `right lifting property'.

Given a homologically graded space, $V$, define the suspension, $\Sigma V$, to be the homologically graded space using the convention $(\Sigma V)_i =V_{i-1}$. In the cohomological setting the suspension is defined by $(\Sigma V)^i =V^{i+1}$. When dealing with objects that are endowed with a topology (such as those that are pseudo-compact) taking the dual will be understood to mean taking the topological dual. In more detail, this is the functor that takes an object to the set of continuous morphisms from it to the ground field. Therefore, since any continuous functional on $V^*$ must factor through a finite dimensional quotient, one has (by the well known property of finite dimensional spaces) the continuous functional corresponds to an element of $V$, i.e.~within this paper it will always be the case that $V^{**}\cong V$. Applying the functor of linear discrete (or topological) duality takes homologically graded spaces to cohomologically graded ones, and vice versa---more precisely, denoting the dual by an asterisk, it can be seen that $(V_i)^*=(V^*)^{i}$. Note, a homologically graded space can be made into a cohomologically graded one (and vice versa), by setting $V_i=V^{-i}$. Moreover, $\Sigma V^*$ will be used to denote $\Sigma (V^*)$, and with this notation there exists an isomorphism $(\Sigma V )^*\cong \Sigma^{-1} V^*$.

All Lie algebras will be given the homological grading with lower indices, whereas commutative algebras will be given the cohomological grading with upper indices. There is, however, an important exception to this convention: the tensor product of a homologically graded Lie algebra with the cohomologically graded cdga $\Omega$ of the Sullivan-de Rham forms. In this context, $\Omega$ is considered as homologically graded using the above relations (i.e.~$\Omega_i:=\Omega^{-i}$ for each $i\geq 0$), ensuring the tensor product is a homologically graded curved Lie algebra. This convention is relevant in Section \ref{sec_rat_homo}.

Although many of the Lie algebras in this paper (namely the curved Lie algebras) will not necessarily be complexes, they resemble them and possess an odd derivation often referred to (abusively) as the differential. In the homological grading, the derivation possessed by a curved Lie algebra has degree $-1$. Given any homogeneous element, $x$, of some given graded algebra its degree is denoted by $|x|$. Therefore, in the homological setting a Maurer-Cartan element is of degree $-1$ and the curvature element is of degree $-2$.

\section{The category of curved Lie algebras}\label{sec_curvedLie}

In this section the category of pseudo-compact curved Lie algebras with curved morphisms will be defined and some of its basic properties are discussed. Later, in Section \ref{sec_model_cat}, the category of pseudo-compact curved Lie algebras with curved morphisms will be shown to possess a model structure. Further, a pair of adjoint functors will be defined in Section \ref{sec_adjoint} that induce a Quillen equivalence between this model category and the model category of unital cdgas given by Hinich \cite{hinich_homo_alg}, this is the content of Theorem \ref{thm_main}.

\begin{definition}\label{def_curvedLie}
A curved Lie algebra is the triple $(\mathfrak{g},d,\omega)$ where $\mathfrak{g}$ is a graded Lie algebra, $d$ is a derivation of $\mathfrak{g}$ with $|d|=-1$ (known as the differential), and $\omega\in\mathfrak{g}$ with $|\omega |=-2$ (known as the curvature) such that:
\begin{itemize}
\item $d\circ d (x)=[\omega , x]$, for all $x\in\mathfrak{g}$;
\item $d\omega = 0$.
\end{itemize}
\end{definition}

\begin{remark}
Notice that the term `differential' is an abuse of notation as the differential of a curved Lie algebra need not square to zero. Accordingly, some authors prefer to use the term predifferential. Within this paper, the term `differential' will be used, despite the abuse of notation.
\end{remark}

A curved Lie algebra with zero curvature is precisely a dgla. As with dglas, Maurer-Cartan elements play a key role in the theory of curved Lie algebras.

\begin{definition}
A MC (Maurer-Cartan) element of a curved Lie algebra, $(\mathfrak{g},d,\omega )$  is an element $\xi\in\mathfrak{g}$ with $|\xi |=-1$ such that the Maurer-Cartan equation,
\[
\omega + d\xi + \frac{1}{2}[\xi,\xi]=0,
\]
is satisfied.
\end{definition}

\begin{remark}
If $\omega=0$ (i.e.~$\mathfrak{g}$ is a differential graded Lie algebra) then the classical MC equation is recovered:
\[
d\xi +\frac{1}{2}[\xi,\xi]=0.
\]
\end{remark}

It will be common practice within this paper to shorten the notation of a curved Lie algebra, $(\mathfrak{g},d_\mathfrak{g},\omega_\mathfrak{g})$, to only its underlying graded Lie algebra, $\mathfrak{g}$, where there is no ambiguity in doing so. In the shortened case, the differential and curvature will be denoted by the obvious subscript.

\begin{definition}
A curved Lie algebra is said to be pseudo-compact if it is an inverse limit of finite dimensional, nilpotent curved Lie algebras.
\end{definition}

Here it is important to note that there exists a natural topology on any given pseudo-compact curved Lie algebra (or more generally pseudo-compact vector space) generated by subspaces of finite codimension, which is induced by taking the inverse limit. Therefore, the bracket and differential of a pseudo-compact vector space will be continuous with respect to this topology. Moreover, when considering morphisms between pseudo-compact spaces they will always be assumed to continuous with respect to this topology. In particular, when taking the dual of such spaces it will be understood that it is the topological dual being taken as opposed to the linear discrete one. For more details on pseudo-compact algebras see \cite[Section 3]{gabriel}, \cite[Appendix]{keller_yang}, and \cite[Section 4]{vandenbergh}.

\begin{remark}
The definition of a pseudo-compact curved Lie algebra is similar to the definition of a pronilpotent Lie algebra, but more restrictive since every Lie algebra in the inverse system is assumed to be finite dimensional.
\end{remark}

\begin{proposition}\label{prop_completed_tensor_Lie_structure}
Take a pseudo-compact curved Lie algebra, $(\mathfrak{g},d_\mathfrak{g},\omega_\mathfrak{g})$, where $\mathfrak{g}=\varprojlim_i \mathfrak{g}_i$, and a unital cdga, $A$, both homologically graded\footnote{Note this is against the usual convention of the paper, but such an issue is easily rectified (as explained in the notations and conventions).}. The completed tensor product
\[
\mathfrak{g}\hat{\otimes}A=\varprojlim_i \mathfrak{g}_i \otimes A
\]
possesses a well defined pseudo-compact curved Lie algebra structure given as follows: the curvature is defined by $\omega_\mathfrak{g}\hat{\otimes} 1$; the differential is defined on elementary tensors by $d(x \hat{\otimes} a)=d_\mathfrak{g}x \hat{\otimes}a + (-1)^{|x|} x \hat{\otimes} d_A a$; and the bracket is defined on elementary tensors by $ [x\hat{\otimes}a,y\hat{\otimes}b]=[x,y]\hat{\otimes}(-1)^{|y||a|}ab$.
\end{proposition}
\begin{proof}
This is standard and straightforward.
\end{proof}

It will be common for the adjective `completed' to be dropped and $\hat{\otimes}$ to be referred to as the tensor product.

\begin{definition}
A curved morphism of curved Lie algebras is defined to be the pair
\[
(f,\alpha)\colon(\mathfrak{g},d_\mathfrak{g},\omega_\mathfrak{g})\to(\mathfrak{h},d_\mathfrak{h},\omega_\mathfrak{h}),
\]
where $f\colon\mathfrak{g}\to\mathfrak{h}$ is a morphism of graded Lie algebras and $\alpha\in\mathfrak{h}$ with $|\alpha |=-1$ such that:
\begin{itemize}
\item $d_\mathfrak{h} f(x)=f(d_\mathfrak{g} x) - [\alpha , f(x)]$, for all $x\in\mathfrak{g}$;
\item $\omega_\mathfrak{h}=f(\omega_\mathfrak{g}) - d_\mathfrak{h} \alpha - \frac{1}{2}[\alpha,\alpha]$.
\end{itemize}

The image of an element $x\in\mathfrak{g}$ under the action of the curved morphism $(f,\alpha )$ is defined to be $f(x) + \alpha\in\mathfrak{h}$.

The composition of two curved morphisms, $(f,\alpha )$ and $(g,\beta )$, (when such a composition exists) is defined as follows:
\[
(f,\alpha)\circ(g,\beta)=(f\circ g,\alpha + f(\beta)).
\]

A morphism with $\alpha=0$ is said to be strict.
\end{definition}

\begin{remark}
Note that a curved morphism will map $0_\mathfrak{g} \mapsto \alpha$. In fact, a curved morphism $(f,\alpha)$ is equivalent to the composition $(id,\alpha)\circ (f,0)$.
\end{remark}

\begin{remark}
In the case of a strict morphism it can be readily seen that the morphism is simply a graded Lie algebra morphism that respects the differentials and the image of the curvature of the domain is the curvature of the codomain. These morphisms are exactly those of \cite{chuang_laz_mannan}. The $\alpha$ part of a curved morphism can, therefore, be seen to act as an obstruction to the differentials commuting with the graded Lie algebra morphism and to the graded Lie algebra morphism preserving the curvature.
\end{remark}

Just as with dgla morphisms, curved Lie algebra morphisms preserve MC elements.

\begin{proposition}
Given a curved Lie algebra morphism 
\[
(f,\alpha)\colon(\mathfrak{g},d_\mathfrak{g},\omega_\mathfrak{g})\to(\mathfrak{h},d_\mathfrak{h},\omega_\mathfrak{h})
\]
and a MC element $\xi\in\mathfrak{g}$, the element $f(\xi)+\alpha \in\mathfrak{h}$ solves the MC equation.
\end{proposition}
\begin{proof}
A simple application of the definitions.
\end{proof}

\begin{proposition}
Given a morphism of curved Lie algebras, $(f,\alpha)\colon(\mathfrak{g},d_\mathfrak{g},\omega_\mathfrak{g})\to(\mathfrak{h},d_\mathfrak{h},\omega_\mathfrak{h})$, there exists an inverse morphism $(f,\alpha)^{-1}\colon(\mathfrak{h},d_\mathfrak{h},\omega_\mathfrak{h})\to(\mathfrak{g},d_\mathfrak{g},\omega_\mathfrak{g})$ such that:
\begin{itemize}
\item $(f,\alpha )\circ (f,\alpha )^{-1}=(\id_\mathfrak{g},0)$, and
\item $(f,\alpha )^{-1}\circ (f,\alpha )=(\id_\mathfrak{h},0)$,
\end{itemize}
if, and only if, $f$ is an isomorphism of graded Lie algebras. Further, given an inverse graded Lie algebra morphism $f^{-1}$ of $f$, the inverse of $(f,\alpha )$ is given by $(f^{-1},-f^{-1}(\alpha))$.
\end{proposition}
\begin{proof}
If the graded Lie algebra morphism $f$ is invertible, then clearly $(f^{-1},-f^{-1}(\alpha))$ gives a two sided inverse of the morphism $(f,\alpha )$. Conversely, if $(f,\alpha )$ is invertible with inverse $(f,\alpha )^{-1}=(g,\beta )$ then
\begin{itemize}
\item $(f,\alpha )\circ (f,\alpha )^{-1}=(f\circ g,f(\beta) +\alpha)=(\id_\mathfrak{g},0)$, and
\item $(f,\alpha )^{-1}\circ (f,\alpha )=(g\circ f,g(\alpha)+\beta)=(\id_\mathfrak{h},0)$.
\end{itemize}
From these equations it is clear that $g$ must be a two sided graded Lie algebra inverse for $f$. Additionally, it can easily be seen that $\beta=-g(\alpha )$.
\end{proof}

It is important to note that a curved Lie algebra may be isomorphic to one with zero curvature (i.e.~a dgla). To see this, let $\operatorname{ad}_\xi (-)=[\xi ,-]$ be the adjoint action and take the curved isomorphism
\[
(\id,-\xi)\colon(\mathfrak{g},d,\omega)\to\left(\mathfrak{g},d+\operatorname{ad}_\xi,\omega+d\xi+\frac{1}{2}[\xi,\xi]\right),
\]
which has inverse $(\id,\xi)$. The curvature of the codomain is zero precisely when $\xi$ is a MC element of $(\mathfrak{g},d,\omega)$. Whence, the resulting curved Lie algebra will have zero curvature if, and only if, the element $\xi$ belongs to the set of MC elements of $(\mathfrak{g},d,\omega )$. In fact, these morphisms actually correspond to twisting by the element $\xi$, and such a twisting is denoted $\mathfrak{g}^\xi$. Twists of dgla are one way in which curved Lie algebras arise in mathematics. For more details see \cite{braun,chuang_laz}, but note the notion of a curved morphism is not used in either citation.

\begin{definition}
The category whose objects are pseudo-compact curved Lie algebras and morphisms are given by the continuous (with respect to the topology induced in taking the inverse limit) curved morphisms between them will be referred to as the category of pseudo-compact curved Lie algebras and will be denoted by $\curvedLie$.
\end{definition}

In \cite{block_laz} an antiequivalence between the categories of pseudo-compact Lie algebras and conilpotent Lie coalgebras is given using the functor of linear duality. Note that in loc.~cit.~pseudo-compact Lie algebras were referred to as pronilpotent Lie algebras---a term that is not ideal, cf.~\cite[Remark 7.2]{laz_markl}. Despite this, pseudo-compact curved Lie algebras are pronilpotent in the classical sense, i.e.~they are isomorphic to an inverse limit of nilpotent curved Lie algebras. 

\begin{proposition}\label{prop_pronilpotent}
For $\mathfrak{g}\in\curvedLie$, let $\mathfrak{g}=F_1 \mathfrak{g}\supseteq F_2 \mathfrak{g} \supseteq \ldots$ denote the lower central series. Then $\mathfrak{g}=\varprojlim_i \frac{\mathfrak{g}}{F_i \mathfrak{g}}$.
\end{proposition}
\begin{proof}
For $\mathfrak{g}\in\curvedLie$ one has, by definition, that $\mathfrak{g}=\varprojlim_n \mathfrak{g}_n$, where each $\mathfrak{g}_n$ is a finite dimensional curved Lie algebra. Now, the filtered limit of finite dimensional vector spaces is exact (since any filtered system of finite dimensional vector spaces satisfies the Mittag-Leffler condition see \cite[Section 3.5]{weibel}) and so by appropriate usage of exact sequences one arrives at $F_i (\mathfrak{g}_n) \cong \varprojlim_n F_i (\mathfrak{g}_n)$, for $i\geq 1$, and $\mathfrak{g}/F_i \mathfrak{g} \cong \varprojlim_n \mathfrak{g}_n/F_i (\mathfrak{g}^n)$. By definition, the finite dimensional curved Lie algebra $\mathfrak{g}_n$ is nilpotent, and so $\varprojlim_i \mathfrak{g}_n/F_i (\mathfrak{g})\cong\mathfrak{g}_n$. Whence,
\[
\varprojlim_i \mathfrak{g}/F_i\mathfrak{g} \cong \varprojlim_i \varprojlim_n \mathfrak{g}_n /F_i (\mathfrak{g}^n) \cong \varprojlim_n \varprojlim_i \mathfrak{g}_n / F_i (\mathfrak{g}^n) \cong \varprojlim_n \mathfrak{g}_n =\mathfrak{g}.\qedhere
\]
\end{proof}

\subsection{Filtrations}

This section briefly discusses filtered curved Lie algebras and some basic properties of filtrations that will be insisted upon. Filtrations for cdgas are also (even more briefly) discussed. For more details regarding filtrations consult \cite[Section 5.4]{weibel}. It should be noted that a more general notion of filtration is treated in loc.~cit.

\subsubsection{Filtrations of curved Lie algebras}

Filtrations play an important role in the model structure for curved Lie algebras (see Section \ref{sec_modelCat_main}) and are fundamental in defining the weak equivalences. The key details of the filtrations of curved Lie algebras used within this paper are briefly discussed here.

\begin{definition}
A curved Lie algebra, $(\mathfrak{g},d,\omega )$, is said to be filtered when equipped with a (descending) filtration denoted by $\lbrace F_i \mathfrak{g}\rbrace_{i\in \mathbb{N}}$ corresponding to a tower
\[
\mathfrak{g} = F_1 \mathfrak{g} \supseteq F_2 \mathfrak{g} \supseteq F_3 \mathfrak{g} \supseteq \dots
\]
of subspaces $F_i \mathfrak{g}$ for all $i\in\mathbb{N}$ such that the filtration respects the bracket and differential, i.e.~
\[
[F_i \mathfrak{g} , F_j \mathfrak{g}]\subseteq F_{i+j} \mathfrak{g}
\;\;\mathrm{and}\;\;
d(F_i \mathfrak{g})\subseteq F_i \mathfrak{g}.
\]
\end{definition}

Notice here that only positively indexed filtrations of curved Lie algebras are considered. This is an important point as even the seemingly harmless step of allowing non-negatively indexed filtrations will drastically change the meaning of a filtered quasi-isomorphism (see Definition \ref{def_qiso}).

\begin{definition}
Given a filtration $\lbrace F_i \mathfrak{g} \rbrace_{i\in \mathbb{N}}$ of a curved Lie algebra $\mathfrak{g}$, the associated graded algebra, denoted $\gr_F \mathfrak{g}$, is the algebra given by the sum
\[
\bigoplus_{i\in \mathbb{N}} \frac{F_i \mathfrak{g}}{F_{i+1} \mathfrak{g}}.
\]
\end{definition}

\begin{proposition}
The associated graded algebra of a filtered curved Lie algebra inherits the bracket and differential.
\end{proposition}
\begin{proof}
It is a well-known and straightforward consequence of the fact that the filtration respects the bracket and differential.
\end{proof}

\begin{definition}
Given a filtered pseudo-compact curved Lie algebra, $\mathfrak{g}$, its filtration is said to be complete if
\[
\mathfrak{g}=\varprojlim_i \frac{\mathfrak{g}}{F_i \mathfrak{g}}.\
\]
\end{definition}

\begin{remark}
The completeness condition is precisely that of $\mathfrak{g}$ being pronilpotent, i.e.~an inverse limit of nilpotent curved Lie algebras.
\end{remark}

\begin{definition}
A filtration is said to be Hausdorff if $\bigcap_{i\geq 1} F_i A=0$.
\end{definition}

\begin{definition}
Let $\mathfrak{g}$ be a pseudo-compact curved Lie algebra. The filtration induced by the lower central series of $\mathfrak{g}$ is the filtration inductively given by
\[
F_1 \mathfrak{g}=\mathfrak{g}, \quad F_2 \mathfrak{g}=[F_1 \mathfrak{g}, \mathfrak{g}], \quad F_3 \mathfrak{g}=[F_2 \mathfrak{g},\mathfrak{g}], \quad \dots, \quad F_{i+1} \mathfrak{g} = [F_i \mathfrak{g},\mathfrak{g}],\quad \dots
\]
\end{definition}

\begin{proposition}\label{prop_lower_central_complete}
The filtration given by the lower central series of a curved Lie algebra (above) respects the bracket and differential. Moreover, the filtration is complete and Hausdorff.
\end{proposition}
\begin{proof}
The fact that the filtration respects the differential and bracket are quick checks. The final statement is a direct consequence of Proposition \ref{prop_pronilpotent}.
\end{proof}

\begin{definition}
Complete (and thus Hausdorff) filtrations that respect the bracket and differential subject to the additional condition that $d^2=0$ in the associated graded objects (i.e.~the associated graded algebras are complexes) are said to be admissible.
\end{definition}

\begin{proposition}\label{prop_canonical_filtration}
The associated graded curved Lie algebra of a filtered curved Lie algebra, where the filtration is given by the lower central series, will be a true complex, i.e.~$d^2=0$ on the associated graded.
\end{proposition}
\begin{proof}
To see that $d^2=0$ on the associated graded notice that the adjoint action of the curvature will increase the filtration degree whereas the action of $d$ fixes the filtration degree.
\end{proof}

Combining Proposition \ref{prop_lower_central_complete} and Proposition \ref{prop_canonical_filtration}, one immediately sees the following.

\begin{proposition}
The filtration given by the lower central series of a curved Lie algebra is an admissible filtration.\qed
\end{proposition}

\begin{definition}
A morphism of filtered curved Lie algebras is a morphism of curved Lie algebras that is compatible with the filtrations in the sense that $(f,\alpha ) F_i \mathfrak{g} \subseteq F_i \mathfrak{h}$.
\end{definition}

\begin{proposition}
A morphism of filtered curved Lie algebras induces a well-defined morphism of the associated graded objects.
\end{proposition}
\begin{proof}
A straightforward consequence of the definitions.
\end{proof}

The objects of $\curvedLie$ do not form complexes (as it is not necessary that $d^2=0$) and as such there is no natural definition of a quasi-isomorphism. It is, however, possible to define the notion of a filtered quasi-isomorphism. Filtered quasi-isomorphisms are of particular importance when defining the weak equivalences of the model structure in Section \ref{sec_model_cat}.

\begin{definition}\label{def_qiso}
Let $\mathfrak{g},\mathfrak{h}\in\curvedLie$ be endowed with admissible filtrations, both denoted by $F$. A filtered morphism $(f,\alpha)\colon\mathfrak{g}\to\mathfrak{h}$ in $\curvedLie$ is said to be a filtered quasi-isomorphism if the induced morphism $\gr_F (f,\alpha)\colon\gr_F \mathfrak{g}\to \gr_F \mathfrak{h}$ is a quasi-isomorphism of dgla, i.e.~$gr_F (f,\alpha)$ induces an isomorphism on the level of homology.
\end{definition}

\begin{remark}
It is important that the filtrations be positively indexed in the definition of a filtered quasi-isomorphism. Even the seemingly innocent adaptation to allow non-negatively indexed filtrations could allow for the case when the filtrations are concentrated solely in degree $0$ and in this case a filtered quasi-isomorphism would be no different to a quasi-isomorphism. This is a problem because it is necessary that the notion of a weak equivalence in the model category of pseudo-compact curved Lie algebras be finer than a quasi-isomorphism, and so only positively indexed filtrations are considered.
\end{remark}

It should be remarked here that no claim is made about the closure of filtered quasi-isomorphisms under composition.

\subsubsection{Filtrations of cdgas}

The above constructions for curved Lie algebras have---where applicable---a counterpart for cdgas whose definitions follow easily from those given above. In this paper, the filtrations of cdgas are different and subject to less restrictions than the filtrations of curved Lie algebras. These differences reflect the contravariant duality between cdgas and curved Lie algebras. The only assumptions made about filtrations of cdgas in this paper are the following: they are ascending and are cocomplete. No other assumptions are made about the filtrations on cdgas. In particular, there is no restriction on the indexing, nor is it necessary that a filtration of a cdga be complete.

\begin{definition}
An ascending filtration of a cdga $A$
\[
F_1 A \subseteq F_2 A \subseteq F_3 A \subseteq \dots
\]
is called cocomplete (or exhaustive) if $A=\varinjlim F_i A$.
\end{definition}

It will only be necessary to consider filtrations of cdgas in a couple of proofs, and in particular the proof of Lemma \ref{lemma_Key}. This is quite different to the fundamental role admissible filtrations play in defining the weak equivalences of pseudo-compact curved Lie algebras.

\section{Analogues of the Chevalley-Eilenberg and Harrison complexes}\label{sec_adjoint}

The category of unital cdgas with the standard cdga morphisms will be denoted by $\mathscr{A}$. Within this section a pair of contravariant functors will be defined (extending those of \cite{laz_markl} and influenced by work of \cite{positselski}). These functors will be shown to be adjoint, providing the base for the Quillen equivalence proven in Theorem \ref{thm_main}.

In the following, a contravariant functor $\hatL \colon \mathscr{A}\to\curvedLie$ will be constructed. Several ideas of Positselski \cite{positselski} feature heavily in this construction. First, note that given a unital cdga, $A \in \mathscr{A}$, the underlying field is a subspace as $k=k\cdot 1\subseteq A$, and it is, therefore, possible to take a linear retraction $\epsilon \colon A\to k$ (which may not commute with the differential or multiplication). Let $A_+$ denote the kernel of $\epsilon$. As vector spaces it is evident that $A= A_+\oplus k$. Note, if $A$ is augmented then $\epsilon$ can be chosen to be an augmentation (i.e.~$\epsilon$ is a morphism of dg algebras) and hence $A_+$ is the augmentation ideal, i.e.~the decomposition holds on the level of cdga. Since $\epsilon$ is not necessarily a morphism of dg algebras the differential $d\colon A_+\to A$ and the multiplication $m\colon A_+\otimes A_+ \to A$ can be split as $d=(d_+,d_k)$ and $m=(m_+,m_k)$, where\hfill
\vspace{11pt}

\noindent
\begin{minipage}{0.48\textwidth}
\[
d_+\colon A_+\to A_+;
\]
\[
d_k\colon A_+\to k;\;\mathrm{and}
\]
\end{minipage}
\begin{minipage}{0.48\textwidth}
\[
m_+\colon A_+ \otimes A_+ \to A_+;
\]
\[
m_k\colon A_+ \otimes A_+ \to k.
\]
\end{minipage}\hfill
\vspace{11pt}

\begin{definition}\label{def_harrison_1}
Given a unital cdga $A \in \mathscr{A}$ and a linear retraction $\epsilon \colon  A\to k$ with kernel $A_+$, let $\hatL(A)$ denote the pseudo-compact curved Lie algebra given by $(\hat{L}\Sigma A^*_+ ,d^*_+ + m^*_+,d^*_k + m^*_k) $. Here the derivations $d^*_+$ and $m^*_+$ are the extensions of the duals to the whole curved Lie algebra (given by the Leibniz rule); the notation used here (somewhat abusively) is the same for both.
\end{definition}

\begin{remark} The pseudo-compact curved Lie algebra $\hatL (A)$ has zero curvature if the linear retraction $\epsilon$ is a morphism of unital cdga, i.e.~if $A$ is augmented with augmentation $\epsilon$. This is because the $k$ parts of the morphisms vanish.
\end{remark}

One could consider the category of unital cdgas, $\mathscr{A}$, as the category of pairs $(A,\epsilon )$, where $A$ is a unital cdga and $\epsilon\colon A\to k$ is a retraction. This is not necessary as the following result shows.

\begin{proposition}
Given a unital cdga $A$, the pseudo-compact curved Lie algebra $(\hat{L}\Sigma A^*_+ ,d^*_+ + m^*_+,d^*_k + m^*_k)$ depends on the choice of linear retraction $\epsilon\colon A\to k$ up to isomorphism.
\end{proposition}
\begin{proof}
A different choice of retraction is given by $\epsilon'(b)=\epsilon(b)+x(b)$, where $x\in A^*_+$ has degree $0$. This leads to the isomorphism of pseudo-compact curved Lie algebras $(\id,x):\hat{L} (\Sigma A^*_+)\to\hat{L} (\Sigma A^*_+)$ (or a twisting), since $x$ will have degree minus one in $\hat{L} (\Sigma A^*_+)$.
\end{proof}

Let $A,B\in\mathscr{A}$, and $A_+$ and $B_+$ be the kernels of a pair of linear retractions on $A$ and $B$, respectively. Given a morphism $f\colon A\to B$ belonging to $\mathscr{A}$ the linear morphism $f\colon A_+\to B$ can be split as $f=(f_+,f_k)$ where
\[
f_+\colon A_+\to B_+,\;\mathrm{and}
\]
\[
f_k\colon A_+\to k.
\]
Clearly the dual morphism $f^*_+$ can be extended to a graded Lie algebra morphism $f^*_+\colon \hatL (B)\to\hatL (A)$, denoted the same by an abuse of notation. Additionally, notice that it is possible to consider the dual morphism $f^*_k$ as a degree $-1$ element of $\hatL (A)$. With these observations in mind the following proposition is made.

\begin{proposition}\label{def_harrison_2}
Given a morphism, $f\colon A\to B$ of $\mathscr{A}$, the morphism
\[
(f^*_+,f^*_k)\colon \hatL (B)\to \hatL (A))
\]
constructed above is a well defined curved Lie algebra morphism.
\end{proposition}
\begin{proof}
The proof amounts to chasing the definitions.
\end{proof}

\begin{remark}
As with objects, the morphism $\mathcal{L}(f)$ depends on the retractions chosen for the cdgas $A$ and $B$. If either retraction is changed, however, then the obtained morphism $\mathcal{L}(f)$ is changed by a pre- and/or post-composition with an isomorphism.
\end{remark}

\begin{definition}\label{def_harrison}
Let $\hatL \colon  \mathscr{A}\to\curvedLie$ be the contravariant functor that sends a cdga $A$ to $\hatL (A)$ as in Definition \ref{def_harrison_1} and sends a morphism $f$ to $(f^*_+,f^*_k)$ as in Proposition \ref{def_harrison_2}.
\end{definition}

It will be shown that the functor $\hatL$ provides one half of an adjoint pair and so to complete the pair of functors it is necessary to describe a contravariant functor going in the reverse direction.

\begin{definition}\label{def_CE}
The contravariant functor $\C \colon  \curvedLie \to \mathscr{A}$ is given by the following. The underlying graded space is given by the symmetric algebra of the suspension of the continuous dual of $\mathfrak{g}$, i.e.~$S\Sigma\mathfrak{g}^*$. This becomes a unital cdga with the concatenation product and the differential made of three parts coming from the duals of the curvature, the differential and the bracket of $\mathfrak{g}$ made into derivations and extended via the Leibniz rule, cf.~\cite{chuang_laz_mannan}.

Given a morphism $(f,\alpha)\colon \mathfrak{g}\to\mathfrak{h}$ in $\curvedLie$ associate to it the morphism 
\[
\C(f,\alpha)\colon \C(\mathfrak{h})\to \C(\mathfrak{g})
\]
given by $\Sigma(f^*\oplus\alpha^*)\colon \Sigma\mathfrak{h}^*\to \Sigma\mathfrak{g}^*$ extended as a morphism of cdgas.
\end{definition}

\begin{remark}
The functor $\C$ is an analogue of the Chevalley-Eilenberg construction in homological algebra.
\end{remark}

In Definition \ref{def_CE}, the uncompleted symmetric algebra is taken and not the completed one; this is because $\mathfrak{g}$ is already complete in some sense.

If $(\mathfrak{g},d,0)$ has zero curvature then the cdga obtained by applying the contravariant functor $\C$ is augmented, cf.~\cite{laz_markl} since it is then precisely the same construction. It can, therefore, be understood that the curvature acts as an obstruction for the cdga $\C(\mathfrak{g})$ to be augmented, because the natural choice for augmentation fails to be a dg algebra morphism. More precisely, the part of the differential coming from the curvature maps into $k$ and not into $\mathfrak{g}$.

Another reason to see why an augmentation often fails to arise in the curved setting is that a MC element for a dgla corresponds to an augmentation of $S\Sigma\mathfrak{g}^*$. In the uncurved case the zero element is always a MC element and there is always an augmentation. In the case of a curved Lie algebra, however, there need not be any solutions to the MC equation.

\begin{proposition}\label{prop_adjoint}
The contravariant functor $\C \colon \curvedLie\to\mathscr{A}$ is right adjoint to the contravariant functor $\hatL \colon \mathscr{A}\to\curvedLie$.
\end{proposition}
\begin{proof}
In order to prove the proposition it is sufficient to exhibit the following isomorphism:
\[
\Hom_\mathscr{A} (\C (\mathfrak{g}),A)\cong \Hom_{\curvedLie} (\hatL (A),\mathfrak{g}),
\]
for any curved Lie algebra $\mathfrak{g}$ and any unital cdga $A$. To this end, assume $f\colon S\Sigma\mathfrak{g}^*\to A$ is a morphism of unital cdgas. This morphism is uniquely determined by the linear morphism $f\colon \Sigma\mathfrak{g}^*\to A$, which in turn defines $f_+$ and $f_k$ since $A=A_+\oplus k$. Dualising, the linear morphisms
\begin{align*}
f^*_+ &\colon \Sigma A^*_+\to\mathfrak{g},\; \mathrm{and}\\
f^*_k &\colon \Sigma k \to \mathfrak{g}
\end{align*}
are obtained.

By extending $f^*_+\colon \Sigma A_+\to\mathfrak{g}$ as a graded Lie algebra morphism to $\hat{L}\Sigma A_+$ and combining it with $f^*_k$ the curved Lie algebra morphism $(f^*_+,f^*_k)\colon \hatL(A)\to\mathfrak{g}$ is obtained. It is straightforward, although slightly tiresome, to do the calculations. Hence, one side of the adjunction is proven. Now, assume that $(f,\alpha )\colon \hatL (A)\to \mathfrak{g}$ is a curved Lie algebra morphism. The graded Lie algebra morphism, $f$, is uniquely determined by the underlying linear morphism $f\colon \Sigma A^*_+\to\mathfrak{g}$, which induces $f_+\colon \Sigma\mathfrak{g}^*\to A_+$: this gives one component. The second comes from first considering $\alpha$ as the morphism $\alpha \colon \Sigma^{-1} k\to\mathfrak{g}$. By dualising this morphism and taking the suspension one obtains $f_k=\Sigma\alpha^*\colon \Sigma\mathfrak{g}\to k$. This construction yields a linear morphism $(f^*_+,f^*_k)\colon \Sigma\mathfrak{g}\to A$ which can be extended to morphism of commutative algebras $\C (\mathfrak{g})\to A$. The final morphism also commutes with the differentials, which is a quick check. Hence the other side of the adjunction has been proven.
\end{proof}

\section{Model category of curved Lie algebras}\label{sec_modelCat_main}

Here it will be demonstrated that the category $\curvedLie$ can be endowed with the structure of a model category with weak equivalences given by filtered quasi-isomorphisms, cf.~Definition \ref{def_qiso}. In addition, this model structure is Quillen equivalent to the model structure for unital cdga given in \cite{hinich_homo_alg}. This equivalence will be shown using similar methods to \cite{laz_markl}. More precisely, the proof of the equivalence will employ the Quillen equivalence that exists upon associative dg local algebras and pseudo-compact curved associative algebras (see \cite{positselski}), as well as the primitive elements and universal enveloping algebra functors. First, though, it must be proven that $\curvedLie$ possesses all small limits and colimits.

\subsection{Limits and colimits}

There does not exist an initial object in the category of pseudo-compact curved Lie algebras with curved morphisms. The closest object to an initial object is the curved Lie algebra freely generated by a single element (the curvature) of degree $-2$ with zero differential, i.e.~$(\hat{L}\langle\omega \rangle,0,\omega)$. There clearly will be a morphism from this object to every other object. Such a morphism, however, is not necessarily unique. Nevertheless, in the category of curved Lie algebras with strict morphisms this object is the initial object. Thus, it is necessary to formally add an initial object to the category $\curvedLie$. From here on let $\curvedLie_*$ denote the category of pseudo-compact curved Lie algebras and curved morphisms with a formal initial object added. The category $\curvedLie_*$ does possess a terminal object, namely the zero curved Lie algebra $(0,0,0)$.

\begin{proposition}
Here the product over a finite set will be described; the general case follows in a straightforward fashion from this description. The product over the set $I=\lbrace i_1, i_2,\dots i_n \rbrace$ indexing pseudo-compact curved Lie algebras, $(\mathfrak{g}_{i_j},d_{i_j},\omega_{i_j})$ for $1\leq j \leq n$, is denoted by $\prod_{i\in I} \mathfrak{g}_i$ and given by the Cartesian product of underlying sets with bracket given by
\[
[(x_{i_1},\ldots ,x_{i_n}),(x'_{i_1},\ldots ,x'_{i_n})]=([x_{i_1},x'_{i_1}]_{i_1},\ldots ,[x_{i_n},x'_{i_n}]_{i_n}),
\]
where $[,]_{i_j}$ is the bracket of $\mathfrak{g}_{i_j}$, differential given by
\[d(x_{i_1},\ldots ,x_{i_n})=(d_{i_1}x_{i_1},\ldots ,d_{i_n}x_{i_n}),
\]
and curvature given by
\[
(\omega_{i_1},\ldots ,\omega_{i_n}).
\]
The projection morphisms are the obvious ones onto each factor, i.e.~$\pi_{i_j}\colon \prod_{i\in I} \mathfrak{g}_i \to\mathfrak{g}_{i_j}$.
\end{proposition}
\begin{proof}
A straightforward check.
\end{proof}

\begin{proposition}
The equaliser of two curved morphisms $(f,\alpha),(g,\beta)\colon \mathfrak{g}\to\mathfrak{h}$ is given by the initial object if $\alpha\neq\beta$ and by $\lbrace x\in\mathfrak{g} \colon f(x)=g(x)\rbrace$ if $\alpha = \beta$.
\end{proposition}
\begin{proof}
If $\alpha\neq\beta$ then $0$ is not in the equaliser, because $f(0)-\alpha\neq g(0) -\beta$. Therefore, the initial object is the only object satisfying the conditions of the equaliser.

If $\alpha = \beta$, then it is a straightforward exercise to show that the space $\lbrace x\in\mathfrak{g} \colon f(x)=g(x)\rbrace$ respects the differential and bracket inherited from $\mathfrak{g}$.
\end{proof}

\begin{proposition}\label{prop_coprod}
The coproduct in the category of pseudo-compact curved Lie algebras is easiest to describe in the binary case: given two pseudo-compact curved Lie algebras, $(\mathfrak{g},d_\mathfrak{g},\omega_\mathfrak{g})$ and $(\mathfrak{h},d_\mathfrak{h},\omega_\mathfrak{h})$, the coproduct $\mathfrak{g}\coprod\mathfrak{h}$ has underlying graded Lie algebra $\mathfrak{g}\ast \mathfrak{h} \ast \hat{L} \langle z \rangle$, where $\ast$ is the free product and $z$ is a formal element of degree minus one. The differential is given by the rules: $d|_\mathfrak{g} = d_{\mathfrak{g}}$, $d|_\mathfrak{h} = d_\mathfrak{h} - ad_z$ and $dz=\omega_\mathfrak{g} - \omega_\mathfrak{h} + \frac{1}{2}[x,x]$. The resulting space has curvature equal to that of $\mathfrak{g}$. The two inclusion morphisms are given by
\[
(id_\mathfrak{g},0)\colon \mathfrak{g} \hookrightarrow \mathfrak{g} \coprod \mathfrak{h}
\]
and
\[
(id_\mathfrak{h},z)\colon \mathfrak{h} \hookrightarrow \mathfrak{g} \coprod \mathfrak{h}.
\]
\end{proposition}
\begin{proof}
It is straightforward to see that $\mathfrak{g}\coprod \mathfrak{h}$ is a well defined pseudo-compact curved Lie algebra. To demonstrate the required universal property, take a pseudo-compact curved Lie algebra $\mathfrak{x}$ and morphisms
\[
(f_\mathfrak{g},\alpha)\colon \mathfrak{g} \to \mathfrak{x}
\] and 
\[
(f_\mathfrak{h},\beta)\colon \mathfrak{h} \to \mathfrak{x}.
\]
Define a morphism $(f,\alpha): \mathfrak{g} \coprod \mathfrak{h} \to \mathfrak{x}$ by $f|_\mathfrak{g}=f_\mathfrak{g}$, $f|_\mathfrak{h}=f_\mathfrak{h}$, and $f(z)=\beta -\alpha$. Again it is straightforward to show that $(f,\alpha )$ is a well defined morphism and the diagram
\begin{center}
\begin{tikzpicture}
\matrix (m) [matrix of math nodes, column sep=5em, row sep=4em]{
 & \mathfrak{x} & \\
\mathfrak{g} & \mathfrak{g}\coprod\mathfrak{h} & \mathfrak{h} \\};
\path[->]
(m-2-1) edge node[above,xshift=-0.5cm] {$(f_\mathfrak{g},\alpha)$} (m-1-2)
(m-2-2) edge node[left,yshift=-0.3cm] {$(f,\alpha )$} (m-1-2)
(m-2-3) edge node[above,xshift=0.5cm] {$(f_\mathfrak{h},\beta )$} (m-1-2)
(m-2-1) edge (m-2-2)
(m-2-3) edge (m-2-2);
\end{tikzpicture}
\end{center}
commutes. Uniqueness of this construction is a quick check.
\end{proof}

\begin{remark}\label{remark_not_exact}
Fix $\mathfrak{h}\in\curvedLie_*$. Given $\mathfrak{g}\in\curvedLie_*$, the functor assigning $\mathfrak{g}\mapsto\mathfrak{g}\coprod\mathfrak{h}$ is not exact. This is easily seen as the terminal object is not preserved.
\end{remark}

\begin{remark}
Since $\coprod$ is a coproduct in the category $\curvedLie_*$ there exists a curved isomorphism
\[
\mathfrak{g}\coprod\mathfrak{h}\cong\mathfrak{h}\coprod\mathfrak{g}.
\]
The isomorphism is strictly curved as the two resulting curved Lie algebras are related by a twist. Explicitly the morphism is given by $(f,z)$, where $f$ is the identity on $\mathfrak{g}$ and $\mathfrak{h}$, and maps $z$ to $-z$. The inverse morphism has the identical action.
\end{remark}

The (categorical) coproduct of Proposition \ref{prop_coprod} is similar to the (non-categorical) disjoint product of \cite{laz_markl}. Informally, the coproduct can be thought of as taking the disjoint union of the two spaces, formally adding a MC base point (i.e.~a solution the Maurer-Cartan equation) and then twisting the copy of $\mathfrak{h}$ with this base point to flatten its curvature.

\begin{proposition}
The coequaliser of two curved morphisms $(f,\alpha), (g,\beta)\colon \mathfrak{g}\to\mathfrak{h}$ is the quotient of $\mathfrak{h}$ by the ideal generated by $f(x)-g(x)$ and $\alpha -\beta$, for all $x\in\mathfrak{g}$.
\end{proposition}
\begin{proof}
A painless check.
\end{proof}

\begin{proposition}
The category $\curvedLie_*$ has all small limits and colimits.
\end{proposition}
\begin{proof}
The category $\curvedLie_*$ has an initial object, a terminal object, all products, all equalisers, all coproducts, and all coequalisers, therefore it has all small limits and small colimits, cf.~\cite[Chapter V]{cat_for_work}.
\end{proof}

\subsection{Duality for associative dg algebras}

It is now necessary to recall the definitions of the cobar constructions in the associative case. These constructions can be found, for example, in \cite{positselski}, where pseudo-compact local associative algebras were studied in the dual setting as conilpotent coassociative coalgebras.

\begin{definition}
A curved associative algebra is a graded algebra with an odd derivation, $d$, called the differential and an element of degree $-2$, $\omega$, called the curvature, such that:
\begin{itemize}
\item $d\circ d =\operatorname{ad}_\omega $;
\item $d(\omega )=0$.
\end{itemize}
\end{definition}

\begin{definition}
The category of associative dg algebras with dg algebra morphisms will be denoted $\mathcal{A}ss$.

A curved associative algebra is said to be pseudo-compact if is isomorphic to an inverse limit of finite dimensional nilpotent curved associative algebras. The category of pseudo-compact local associative curved algebras with continuous curved associative algebra morphisms will be denoted $\widehat{\mathcal{CA}ss}$.
\end{definition}

Note that, just like for $\curvedLie_*$, one has to formally add an initial object to $\widehat{\mathcal{CA}ss}$ for it to possess all limits. The category $\widehat{\mathcal{CA}ss}$ with an initial object formally added will be denoted $\widehat{\mathcal{CA}ss}_*$.

The following is a result of Hinich and Jardine, cf.~\cite{hinich_homo_alg,jardine}.

\begin{theorem}
The category $\mathcal{A}ss$ possesses a model structure with the class of weak equivalences given by quasi-isomorphisms. \qed
\end{theorem}

The following is a result of Positselski, cf.~\cite[Section 9]{positselski}.

\begin{theorem}
The category $\widehat{\mathcal{CA}ss}_*$ possesses a model structure, where the weak equivalences are given by the minimal class of morphisms containing all of the filtered quasi-isomorphisms and satisfying the two out of three property. \qed
\end{theorem}

\begin{definition}
Let $\hat{\mathcal{B}}\colon \mathcal{A}ss\to\widehat{\mathcal{CA}ss}_*$ be the contravariant functor assigning to an associative dg algebra the pseudo-compact associative curved algebra $\hat{\mathcal{B}}(A)$ whose underlying graded algebra is $\hat{\operatorname{T}}\Sigma A^*_+$, where $A_+$ is the kernel of a linear retraction $A\to k$. The differential is induced from the multiplication and differential in the same way as Definition \ref{def_harrison_1}.
\end{definition}

Just as with the functor $\hatL$ given in Section \ref{sec_adjoint}, the resulting pseudo-compact associative curved algebra under the functor $\hat{\mathcal{B}}$ has zero curvature if, and only if, the linear retraction is a true augmentation, i.e.~a dg algebra morphism.

\begin{definition}
Let $\mathcal{B}\colon \widehat{\mathcal{CA}ss}_* \to\mathcal{A}ss$ be the contravariant functor assigning to a pseudo-compact associative curved algebra the associative dg algebra $\mathcal{B}(A)$ whose underlying graded algebra is $\operatorname{T}\Sigma A^*_+$, where $A_+$ is again the kernel of a linear retraction $A\to k$. The differential is induced in the same way as in Definition \ref{def_harrison}.
\end{definition}

\begin{remark}
It is important to notice that given an associative algebra, $A$, the pseudo-compact associative algebra $\hat{\mathcal{B}}(A)$ is in fact a Hopf algebra, since it is a tensor algebra. The space of algebra generators in $\hat{\mathcal{B}}(A)$ contains only primitive elements. Moreover, if the multiplication of the associative algebra $A$ is commutative, then the differential of $\hat{\mathcal{B}}(A)$ maps the space of algebra generators to primitives. Thus, for a cdga $A$, $\hat{\mathcal{B}}(A)$ is a dg Hopf algebra.
\end{remark}

Much like in \cite{laz_markl}, the reason for recalling the definitions of the associative case is that the contravariant functors $\C$ and $\hatL$ can be `embedded' into the following adjunction proven by Positselski \cite{positselski}.

\begin{theorem}
The contravariant functors $\mathcal{B}$ and $\hat{\mathcal{B}}$ are adjoint. Moreover, they induce a Quillen equivalence.
\end{theorem}
\begin{proof}
Pseudo-compact local associative curved algebras are dual to conilpotent coassociative curved coalgebras and thus it follows from \cite[Section 6]{positselski}.
\end{proof}

Let $U\colon \curvedLie_*\to\widehat{\mathcal{CA}ss}_*$ be the universal enveloping algebra functor, given analogously to the classical construction. Let $\operatorname{Prim}\colon \widehat{\mathcal{CA}ss}_*\to\curvedLie_*$ be the primitive elements functor, given analogously to the classical construction.

The contravariant functors $\hatL$ and $\C$ can be seen to fit into the following `commutative diagram' of functors:
\begin{center}
\begin{tikzpicture}
\matrix (m) [matrix of nodes, column sep=2.5em, row sep=4em]{
$\mathcal{A}ss$ & $\widehat{\mathcal{A}ss}_*$ \\
$\mathscr{A}$ & $\curvedLie_*$, \\
};
\path[->]
(m-1-1) edge[transform canvas={yshift=2}] node[above] {$\hat{\mathcal{B}}$} (m-1-2)
(m-1-2) edge[transform canvas={yshift=-2}] node[below] {$\mathcal{B}$} (m-1-1)
(m-1-2) edge[transform canvas={xshift=2}] node[right] {$\operatorname{Prim}$} (m-2-2)
(m-2-2) edge[transform canvas={xshift=-2}] node[left] {$U$} (m-1-2)
(m-2-1) edge[right hook->] node[left] {forgetful} (m-1-1)
(m-2-1) edge[transform canvas={yshift=2}] node[auto] {$\hatL$} (m-2-2)
(m-2-2) edge[transform canvas={yshift=-2}] node[auto] {$\C$} (m-2-1);
\end{tikzpicture}
\end{center}
The following two propositions show the necessary `commutativity' of this diagram for the purposes of this paper, and the proofs follow in a straightforward manner from \cite[Proposition 9.5]{laz_markl}.

\begin{proposition}
Given a unital cdga, $A$, there is a natural isomorphism of pseudo-compact curved Lie algebras $\operatorname{Prim}(\hat{\mathcal{B}}(A))\cong\hatL(A)$. \qed
\end{proposition}

\begin{proposition}\label{prop_bar(env)=CE}
Given a pseudo-compact curved Lie algebra, $\mathfrak{g}$, the differential graded algebras $\mathcal{B}(U(\mathfrak{g}))$ and $\mathcal{C}(\mathfrak{g})$ are quasi-isomorphic. \qed
\end{proposition}

\subsection{Model structure}\label{sec_model_cat}

The category $\curvedLie_*$ will now be endowed with a model structure that will be shown to be Quillen equivalent (via the contravariant functors defined in Section \ref{sec_adjoint}) to the model category of unital cdgas given by Hinich \cite{hinich_homo_alg}. For completeness, first recall the model structure given by Hinich.

\begin{theorem}[Hinich]
The category of unital cdgas is a closed model category where a morphism is
\begin{itemize}
\item a weak equivalence if, and only if, it is a quasi-isomorphism;
\item a fibration if, and only if, it is surjective;
\item a cofibration if, and only if, it has the LLP with respect to all acyclic fibrations. \qed
\end{itemize}
\end{theorem}

\begin{definition}\label{def_model_structure_of_curved_Lie}
A morphism $(f,\alpha)\colon \mathfrak{g}\to\mathfrak{h}$ in $\curvedLie_*$ is called
\begin{itemize}
\item a weak equivalence if, and only if, it belongs to the minimal class of morphisms containing the filtered quasi-isomorphisms and satisfy the two out of three property;
\item a fibration if, and only if, the underlying graded Lie algebra morphism, $f$, is a surjective morphism;
\item a cofibration if, and only if, it has the LLP with respect to all acyclic fibrations.
\end{itemize}
\end{definition}

Some preliminary results will now be discussed before showing that $\curvedLie_*$ is in fact a model category with the above model structure. First it is helpful to look at some useful facts regarding filtered quasi-isomorphisms and the units of the adjunction given in Section \ref{sec_adjoint}.

\begin{proposition}\label{prop_C_preserves_weak}
Given a filtered quasi-isomorphism $(f,\alpha)\colon \mathfrak{g}\to\mathfrak{h}$, the induced morphism
\[
\C(f,\alpha)\colon \C(\mathfrak{h})\to\C(\mathfrak{g})
\]
is a quasi-isomorphism of $\mathscr{A}$. Conversely, given a quasi-isomorphism $g\colon A\to B$ of $\mathscr{A}$ the induced morphism
\[
\hatL(g)\colon \hatL (B)\to\hatL (A)
\]
is a filtered quasi-isomorphism.
\end{proposition}
\begin{proof}
In the first instance, there exist filtrations on $\mathfrak{g}$ and $\mathfrak{h}$ such that $\gr_F (f,\alpha)$ is a quasi-isomorphism. These filtrations induce increasing and cocomplete filtrations upon $\C (\mathfrak{g})$ and $\C(\mathfrak{h})$. Therefore, $\gr_F \C(f,\alpha)$ is a quasi-isomorphism and---since the homology of chain complexes commutes with filtered colimits---$\C (f,\alpha)$ is a quasi-isomorphism.

Now for the converse statement. After applying the contravariant functor $\hatL$, take the filtrations induced by the lower central series. Since the brackets are built freely they preserve quasi-isomorphisms, whence $\hatL (g)$ is a filtered quasi-isomorphism.
\end{proof}

\begin{proposition}
\leavevmode
\begin{itemize}
\item Given a unital cdga, $A$, the morphism $i_A\colon \mathcal{C}\hatL(A)\to A$ is a quasi-isomorphism.
\item Given a pseudo-compact curved Lie algebra, $\mathfrak{g}$, the morphism $i_\mathfrak{g}\colon \hatL\mathcal{C}(\mathfrak{g})\to \mathfrak{g}$ is a filtered quasi-isomorphism.
\end{itemize}
\end{proposition}
\begin{proof}
By Proposition \ref{prop_bar(env)=CE} $\mathcal{C}(\hatL(A))\simeq \mathcal{B}U(\hatL(A))$. Now, $\mathcal{B}U\hatL(A)=\mathcal{B}\hat{\mathcal{B}}(A)$ which is quasi-isomorphic to $A$ by \cite[Section 9]{positselski}, considering $A$ as an associative dg algebra. Thus the first statement is proven.

To prove the second statement, consider the natural filtration by the lower central series on $\mathfrak{g}$ and the filtration it induces upon $\hatL\C (\mathfrak{g})$; denote these two filtrations by $F$. Therefore, it is sufficient to the show that the morphism $\hatL\C (\gr_F \mathfrak{g})\to \gr_F \mathfrak{g}$ is a quasi-isomorphism of dgla---this follows from \cite[Proposition 9.10]{laz_markl}.
\end{proof}

\begin{lemma}\label{lemma_swaps_cofibs_and_fibs}
The contravariant functor $\hatL$ sends fibrations to cofibrations and cofibrations to fibrations.
\end{lemma}
\begin{proof}
Given a fibration $f\colon A\to B$ of unital cdgas, to show that $\hatL(f)$ is a cofibration it is necessary to find a lift in each diagram of the form
\begin{center}
\begin{tikzpicture}
\matrix (m) [matrix of math nodes, column sep=2.5em, row sep=2.5em]
{
\hatL (B) & \mathfrak{g} \\
\hatL (A) & \mathfrak{h}, \\
};
\path[->]
(m-1-1) edge (m-1-2)
(m-1-1) edge node[left] {$\hatL (f)$} (m-2-1)
(m-2-1) edge[dashed] (m-1-2)
(m-2-1) edge (m-2-2)
(m-1-2) edge node[right] {$(\phi , \alpha )$} (m-2-2)
;
\end{tikzpicture}
\end{center}
where the morphism $(\phi ,\alpha )\colon \mathfrak{g}\to\mathfrak{h}$ is an acyclic fibration of pseudo-compact curved Lie algebras. This is equivalent to seeking a lift in each diagram of the following form
\begin{center}
\begin{tikzpicture}
\matrix (m) [matrix of math nodes, column sep=2.5em, row sep=2.5em]
{
\mathcal{C}(\mathfrak{h}) & A \\
\mathcal{C}(\mathfrak{g}) & B. \\
};
\path[->]
(m-1-1) edge (m-1-2)
(m-1-1) edge node[left] {$\C (\phi ,\alpha )$} (m-2-1)
(m-2-1) edge[dashed] (m-1-2)
(m-2-1) edge (m-2-2)
(m-1-2) edge node[right] {$f$} (m-2-2);
\end{tikzpicture}
\end{center}
This follows from the adjunction of the contravariant functors $\hatL$ and $\C$, see Proposition \ref{prop_adjoint}. Now, by assumption, the morphism $f$ is a fibration. Further, by Proposition \ref{prop_C_preserves_weak}, the morphism $\C(\phi ,\alpha )$ is a weak equivalence. Therefore, it suffices show that the morphism $\C(\phi ,\alpha )$ is a cofibration. Let $\lbrace F_i\mathfrak{g} \rbrace_{i\in \mathbb{N}}$ denote the lower central series and let the kernel of $( \phi , \alpha )$ be denoted by $K\subseteq\mathfrak{g}$. Therefore, the tower
\begin{center}
\begin{tikzpicture}
\matrix (m) [matrix of math nodes, column sep=2.5em]
{\mathfrak{h}\cong \mathfrak{g}/(F_1 \mathfrak{g}\cap K) & \mathfrak{g}/(F_2 \mathfrak{g}\cap K) & \mathfrak{g}/(F_3 \mathfrak{g}\cap K) & \cdots\\};
\path[<<-]
(m-1-1) edge node[above] {$\pi_1$} (m-1-2)
(m-1-2) edge node[above] {$\pi_2$} (m-1-3)
(m-1-3) edge node[above] {$\pi_3$} (m-1-4);
\end{tikzpicture}
\end{center}
is obtained. In Proposition \ref{prop_pronilpotent} it was shown that $\mathfrak{g}=\varprojlim_i \mathfrak{g}/F_i \mathfrak{g}$ for any $\mathfrak{g}\in\curvedLie_*$. The limit of the above tower is, therefore, simply $\mathfrak{g}$ and hence the colimit of
\begin{center}
\begin{tikzpicture}
\matrix (m) [matrix of math nodes, column sep=2.3em]
{\C(\mathfrak{h}) & \C(\mathfrak{g}/(F_1 \mathfrak{g}\cap K)) & \C(\mathfrak{g}/(F_2\mathfrak{g}\cap K)) & \C(\mathfrak{g}/(F_3 \mathfrak{g}\cap K)) & \cdots\\};
\path[right hook->]
(m-1-1) edge (m-1-2)
(m-1-2) edge (m-1-3)
(m-1-3) edge (m-1-4)
(m-1-4) edge (m-1-5);
\end{tikzpicture}
\end{center}
is $\C (\mathfrak{g})$. It is thus sufficient to prove that each of the morphisms $\C (\pi_n)$ for $n\in\mathbb{N}$ are cofibrations in $\mathscr{A}$.

Note that for each $n\geq 1$, $\ker(\pi_n)=(F_n \mathfrak{g} \cap K)/(F_{n+1} \mathfrak{g}\cap K)$. From here it is straightforward to see that, just as in \cite[Section 5.2.2]{hinich_stacks}, $\C(\pi_n)$ is a standard cofibration obtained by adding free generators to $\C (\mathfrak{g}/(F_n \mathfrak{g} \cap K))$.

To complete the proof of the statement, one employs the standard fact that cofibrations in unital cdga are monomorphisms and the contravariant functor $\hatL$ sends monomorphisms to epimorphisms, i.e.~fibrations.
\end{proof}

By \cite[Theorem 9.7]{dwyer_spalinski}, once it has been shown that $\curvedLie_*$ is a CMC with the model structure of Definition \ref{def_model_structure_of_curved_Lie}, Theorem \ref{thm_main} will have been proven. To this end, it is first noted that despite the functor $(-)\coprod \mathfrak{h}$, for some fixed $\mathfrak{h}$, not being exact (see Remark \ref{remark_not_exact}) it does have the following redeeming property.

\begin{lemma}
Given a weak equivalence, $(f,\alpha )\colon \mathfrak{g}_1\to\mathfrak{g}_2$ of $\curvedLie_*$, for a fixed $\mathfrak{h}$
\[
(f,\alpha )\coprod (\id_\mathfrak{h}, 0)\colon \mathfrak{g}_1\coprod\mathfrak{h}\to\mathfrak{g}_2\coprod \mathfrak{h}
\]
is a weak equivalence too.
\end{lemma}
\begin{proof}
The morphism $(f,\alpha )$ is a weak equivalence and so there exist filtrations, $F$ and $G$, on $\mathfrak{g}_1$ and $\mathfrak{g}_2$, respectively, such that the induced morphism on the associated graded algebras is a quasi-isomorphism. The filtrations $F$ induces a filtration on the $\mathfrak{g}_1\coprod \mathfrak{h}$ given by
\[
\tilde{F}_i=\twopartdef{F_i\mathfrak{g}_1\coprod\mathfrak{h}}{i=1}{F_i\mathfrak{g}_1\coprod 0}{i>1}.
\]
Likewise, $G$ induces a filtration, $\tilde{G}$, on $\mathfrak{g}_2\coprod \mathfrak{h}$. The associated graded algebras of the coproducts are clearly quasi-isomorphic via the induced morphisms.
\end{proof}

Now, enough auxiliary results have been developed to prove the remaining axioms of a model category, i.e.~the lifting and the factorisation properties. The proofs of the lifting axioms rely on the next lemma which is proven with methods based upon those of Positselski \cite{positselski} which in turn are based upon constructions originally performed by Hinich \cite{hinich_stacks} in the proof of a similar lemma named op.~cit.~the `Key Lemma'.

\begin{lemma}\label{lemma_Key}
Let $A$ be a unital cdga, $\mathfrak{g}$ be a pseudo-compact curved Lie algebra and $f\colon A\to\mathcal{C}(\mathfrak{g})$ be a surjective morphism.
Consider the pushout
\begin{center}
\begin{tikzpicture}
\matrix (m) [matrix of math nodes, column sep=2em, row sep=2em]
{
\hatL\mathcal{C}(\mathfrak{g})	& \hatL(A) \\
\mathfrak{g}									& \hatL (A) \coprod_{\hatL\C (\mathfrak{g})} \mathfrak{g}. \\
};
\path[->]
(m-1-1) edge node[above] {$\hatL(f)$} (m-1-2)
(m-1-1) edge node[left] {$i_\mathfrak{g}$} (m-2-1)
(m-1-2) edge node[right] {$j$} (m-2-2)
(m-2-1) edge (m-2-2);
\end{tikzpicture}
\end{center}
Then the morphism $j\colon \hatL(A)\to \hatL (A) \coprod_{\hatL\C (\mathfrak{g})} \mathfrak{g}$ is a weak equivalence.
\end{lemma}
\begin{proof}
The morphism $j\colon \hatL(A)\to \hatL (A) \coprod_{\hatL\C (\mathfrak{g})} \mathfrak{g}$ is a weak equivalence if there exists filtrations on $\hatL(A)$ and $\hatL (A) \coprod_{\hatL\C (\mathfrak{g})} \mathfrak{g}$ such that $j$ is a filtered quasi-isomorphism. To construct these filtrations, first filter $\mathfrak{g}$ by the natural filtration obtained by the lower central series. Denote this filtration, the filtration induced upon $\C(\mathfrak{g})$, and the filtration induced on $A$ by the pre-images of the surjective morphism $A\to\C (\mathfrak{g})$ by $F$. It is clear that $\gr_F (\hatL(A)\coprod_{\hatL\mathcal{C}(\mathfrak{g})}\mathfrak{g})=\hatL (\gr_F A)\coprod_{\hatL\C (\gr_F \mathfrak{g})} \gr_F \mathfrak{g}$.

Let $n$ denote the positive grading induced by the indexing of the filtration $F$ on $A$. An increasing filtration $G$ on $\gr_F A$ can be given by setting $G_0 \gr_F A = \gr_F A$ and $G_j\gr_F A$ the sum of the components of the ideal $\ker(\gr_F A\to\C ( \gr_F \mathfrak{g}))$ situated in the grading $n\geq j$. The filtration $G$ is locally finite with respect to the grading $n$. Let $G$ also denote the induced filtrations upon $\hatL(\gr_F A)$ and $\hatL(\gr_F A)\coprod_{\hatL \mathcal{C}(\gr_F \mathfrak{g})} \gr_F \mathfrak{g}$. Whence
\[
\gr_G \gr_F \left(\hatL (A) \coprod_{\hatL\C (\mathfrak{g})} \mathfrak{g}\right)= \hatL (\gr_G \gr_F A) \coprod_{\hatL\C (\gr_F \mathfrak{g})} \gr_F \mathfrak{g}.
\]
Therefore, $\gr_G \gr_F A$ and
\[
\gr_G \gr_F (\hatL (A) \coprod_{\hatL\C (\mathfrak{g})} \mathfrak{g})
\]
each have two gradings, $n$ and $j$, coming from the filtrations $F$ and $G$, respectively. A final filtration $H$ is defined by the rules $H_t \gr_G \gr_F A=\bigoplus_{t\geq n,j=0} (\gr_G \gr_F A)_{n,j}\oplus\bigoplus_{n\geq 1,j\geq 0} (\gr_G \gr_F A)_{n,j}$. The filtration $H$ induces a filtration upon $\hatL(\gr_F A)\coprod_{\hatL \mathcal{C}(\gr_F \mathfrak{g})}\mathfrak{g}$ and, therefore,
\[
\gr_H \gr_G \gr_F \left(\hatL (A) \coprod_{\hatL\C (\mathfrak{g})} \mathfrak{g}\right)= \hatL (\gr_H \gr_G \gr_F A) \coprod_{\hatL\C (\gr_F \mathfrak{g})} \gr_F \mathfrak{g}.
\]
Now, $\hatL (\gr_H \gr_G \gr_F A)$ freely generated by $\hatL \C (\gr_F \mathfrak{g})$ and an acyclic differential graded vector space. Moreover,
\[
\gr_H \gr_G \gr_F (\hatL (A) \coprod_{\hatL\C (\mathfrak{g})} \mathfrak{g})
\]
is freely generated by $\gr_F \mathfrak{g}$ and the same acyclic differential graded vector space. Since $\hatL\C (\gr_F \mathfrak{g})\to \gr_F \mathfrak{g}$ is a quasi-isomorphism, so is
\[
\hatL (\gr_H \gr_G \gr_F A)\to \gr_H \gr_G \gr_F (\hatL (A) \coprod_{\hatL\C (\mathfrak{g})} \mathfrak{g})
\]
and hence the statement has been proven.
\end{proof}

\begin{lemma}\label{lemma_factorise}
Given a morphism in the category of curved Lie algebras, $(f,\alpha)\colon \mathfrak{g}\to\mathfrak{h}$, it can be factorised as the composition of
\begin{itemize}
\item a cofibration followed by an acyclic fibration; and
\item an acyclic cofibration followed by a fibration.
\end{itemize} 
\end{lemma}
\begin{proof}
To proceed, first consider the induced morphism $\mathcal{C}(f,\alpha)\colon \mathcal{C}(\mathfrak{h})\to\mathcal{C}(\mathfrak{g})$. Since $\mathscr{A}$ is a model category it is possible factorise this morphism as
\begin{center}
\begin{tikzpicture}
\matrix (m) [matrix of math nodes, column sep=2.5em]
{
\mathcal{C}(\mathfrak{h}) & A & \mathcal{C}(\mathfrak{g}) \\
};
\path[->]
(m-1-1) edge node[above] {$j$} (m-1-2)
(m-1-2) edge node[above] {$p$} (m-1-3);
\end{tikzpicture},
\end{center}
where $j$ is a cofibration and $p$ is a fibration in the category of unital cdga. Further, it is possible to choose either of the morphisms, $j$ and $p$, to be a weak equivalence and doing so will specialise to one of the statements in the lemma. This factorisation in turn induces morphisms
\begin{center}
\begin{tikzpicture}
\matrix (m) [matrix of math nodes, column sep=2.5em]
{
\hatL\mathcal{C}(\mathfrak{g}) & \hatL (A) & \hatL\mathcal{C}(\mathfrak{h}) \\
};
\path[->]
(m-1-1) edge node[above] {$\hatL (p)$} (m-1-2)
(m-1-2) edge node[above] {$\hatL (j)$} (m-1-3);
\end{tikzpicture}
\end{center}
of curved Lie algebras. Therefore, taking the pushout
\begin{center}
\begin{tikzpicture}
\matrix (m) [matrix of math nodes, column sep=2.5em,row sep=2.5em]
{
\hatL\mathcal{C}(\mathfrak{g}) & \hatL (A) \\
\mathfrak{g} & \hatL(A)\coprod_{\hatL\mathcal{C}(\mathfrak{g})} \mathfrak{g}, \\
};
\path[->]
(m-1-1) edge node[above] {$\hatL (p)$} (m-1-2)
(m-1-1) edge node[left] {$i_\mathfrak{g}$} (m-2-1)
(m-1-2) edge node[right] {$(\tilde{\iota},\gamma )$} (m-2-2)
(m-2-1) edge node[below] {$(\iota,\beta )$} (m-2-2);
\end{tikzpicture}
\end{center}
it can be seen that the morphism $(\iota,\beta)\colon \mathfrak{g}\to\hatL(A)\coprod_{\hatL\mathcal{C}(\mathfrak{g})} \mathfrak{g}$ is a cofibration, since it is obtained from a cobase change of the cofibration $\hatL (p)$. Further, there exists a morphism $(\rho ,\epsilon) \colon \hatL(A)\coprod_{\hatL\mathcal{C}(\mathfrak{g})} \mathfrak{g}\to\mathfrak{h}$ coming from the universal property of a pushout:
\begin{center}
\begin{tikzpicture}
\matrix (m) [matrix of math nodes, column sep=2.5em,row sep=2.5em]
{
\hatL\mathcal{C}(\mathfrak{g}) & \hatL (A) & \\
\mathfrak{g} & \hatL(A)\coprod_{\hatL\mathcal{C}(\mathfrak{g})} \mathfrak{g} & \\
 & & \mathfrak{h}. \\
};
\path[->]
(m-1-1) edge node[above] {$\hatL (p)$} (m-1-2)
(m-1-1) edge node[left] {$i_\mathfrak{g}$} (m-2-1)
(m-1-2) edge node[right] {$(\tilde{\iota},\gamma )$} (m-2-2)
(m-2-1) edge node[below] {$(\iota,\beta)$} (m-2-2)
(m-2-1) edge[bend right] node[below] {$(f,\alpha )$} (m-3-3)
(m-1-2) edge[bend left] node[auto] {$i_\mathfrak{h}\circ\hatL (j)$} (m-3-3)
(m-2-2) edge[dashed] node[auto,xshift=-1em] {$(\rho,\epsilon )$} (m-3-3);
\end{tikzpicture}
\end{center}
Hence, $(\rho,\epsilon )$ is a fibration. Moreover, Lemma \ref{lemma_Key} shows that $(\tilde{\iota},\gamma )$ is a weak equivalence.

The composition $(f,\alpha )=(\rho,\epsilon )\circ(\iota,\beta)$ provides the desired decompositions depending upon the choice of weak equivalence in the original decomposition.

For the first statement, it is necessary to show that $(\rho,\epsilon )$ is a weak equivalence. It is, therefore, sufficient to show that $i_\mathfrak{h}\circ\hatL (j)$ is a weak equivalence, since it then follows from the two of three property. To this end, choose $j$ in the original factorisation to be a weak equivalence and the result follows.

For the second factorisation, choose $p$ to be the weak equivalence in the original factorisation. This ensures that $(\iota,\beta )$ is a weak equivalence because the rest of the morphisms in the commutative square $\hatL (p)$, $i_\mathfrak{g}$, and $(\tilde{\iota},\gamma )$ are weak equivalences.
\end{proof}

\begin{remark}
It is possible that the factorisations proven to exist in Lemma \ref{lemma_factorise} could be given more directly. One (standard) method would be to define analogues of the $n$-disk and $n$-sphere in the category $\curvedLie_*$. The author notes he has not investigated this approach.
\end{remark}

One lift is given by the definition of the classes of morphism in the model structure. Therefore, it is only necessary to prove that all cofibrations have the left lifting property with respect to all acyclic fibrations.

\begin{lemma}
Given a commutative diagram of the form

\begin{center}
\begin{tikzpicture}
\matrix (m) [matrix of math nodes,column sep=2.5em,row sep=2.5em]
{
\mathfrak{g} & \mathfrak{a} \\
\mathfrak{h} & \mathfrak{b}, \\
};
\path[->]
(m-1-1) edge (m-1-2)
(m-1-1) edge node[left] {$(f,\alpha )$} (m-2-1)
(m-1-2) edge node[right] {$(\phi ,\beta )$} (m-2-2)
(m-2-1) edge (m-2-2);
\end{tikzpicture}
\end{center}
where $f$ is an acyclic cofibration and $\phi$ is a fibration, there exists a morphism (or lift) $\mathfrak{h}\to\mathfrak{a}$ such that the diagram still commutes, i.e.~acyclic cofibrations have the LLP with respect to all fibrations.
\end{lemma}
\begin{proof}
First, using the proof of Lemma \ref{lemma_factorise} (and borrowing notation from the proof) one has $(f,\alpha )=(\rho,\epsilon )\circ(\iota,\beta)$, where $(\rho,\epsilon )\colon \hatL(A)\coprod_{\hatL\mathcal{C}(\mathfrak{g})} \mathfrak{g}\to\mathfrak{h}$ is an acyclic fibration and $(\iota,\beta) \colon \mathfrak{g}\to\hatL(A)\coprod_{\hatL\mathcal{C}(\mathfrak{g})} \mathfrak{g}$ is an acyclic cofibration. Note that both $(\rho,\epsilon )$ and $(\iota,\beta)$ are acyclic by the $2$ of $3$ property.

Since $(f,\alpha )$ is a cofibration it has, by definition, the LLP with respect to all acyclic fibrations. In particular, the following commutative diagram exists:
\begin{center}
\begin{tikzpicture}
\matrix (m) [matrix of math nodes, column sep=2.5em, row sep=2.5em]
{
\mathfrak{g} & \hatL(A)\coprod_{\hatL\mathcal{C}(\mathfrak{g})} \mathfrak{g} \\
\mathfrak{h} & \mathfrak{h}, \\
};
\path[->]
(m-1-1) edge node[above] {$(\iota ,\beta) $} (m-1-2)
(m-1-1) edge node[left] {$(f,\alpha )$} (m-2-1)
(m-2-1) edge[double, -] (m-2-2)
(m-2-1) edge node[auto,xshift=0.3em,yshift=-0.3em] {$(h,\zeta)$} (m-1-2)
(m-1-2) edge node[right] {$(\rho,\epsilon )$} (m-2-2);
\end{tikzpicture}
\end{center}
which implies that $(f,\alpha )$ is a retract of $(\iota ,\beta)$, as the following diagram shows:
\begin{center}
\begin{tikzpicture}
\matrix (m) [matrix of math nodes, column sep=2.5em, row sep=2.5em]
{
\mathfrak{g} & \mathfrak{g} & \mathfrak{g} \\
\mathfrak{h} & \hatL(A)\coprod_{\hatL\mathcal{C}(\mathfrak{g})} \mathfrak{g}\to\mathfrak{h} & \mathfrak{h}. \\
};
\path[->]
(m-1-1) edge[double,-] (m-1-2)
(m-1-2) edge[double,-] (m-1-3)
(m-1-1) edge node[left] {$(f,\alpha )$} (m-2-1)
(m-2-1) edge node[below] {$(h,\zeta)$} (m-2-2)
(m-2-2) edge node[below] {$(\rho,\gamma)$} (m-2-3)
(m-1-2) edge node[right] {$(\iota,\beta )$} (m-2-2)
(m-1-3) edge node[right] {$(f,\alpha )$} (m-2-3);
\end{tikzpicture}
\end{center}
Therefore, if $(\iota,\beta)$ has the LLP with respect to all fibrations, so must $(f,\alpha )$. Since $(\iota,\beta )$ is obtained by a cobase change of $\hatL(p)$, it suffices to show that $\hatL(p)$ has the LLP with respect to all fibrations. To this end, begin with the following diagram:
\begin{center}
\begin{tikzpicture}
\matrix (m) [matrix of math nodes, column sep=2.5em, row sep=2.5em]
{
\hatL\mathcal{C}(\mathfrak{g}) & \mathfrak{x} \\
\hatL(A) & \mathfrak{y}, \\
};
\path[->]
(m-1-1) edge (m-1-2)
(m-1-1) edge node[left] {$\hatL(p)$} (m-2-1)
(m-2-1) edge (m-2-2)
(m-1-2) edge node[right] {$(\varphi,\kappa )$} (m-2-2);
\end{tikzpicture}
\end{center}
where $(\varphi,\kappa )$ is a fibration of curved Lie algebras. Using the adjunction between the contravariant functors $\mathcal{C}$ and $\hatL$ (Proposition \ref{prop_adjoint}), finding a lift in the above diagram is equivalent to finding a lift in the following diagram:
\begin{center}
\begin{tikzpicture}
\matrix (m) [matrix of math nodes, column sep=2.5em, row sep=2.5em]
{
\mathcal{C}(\mathfrak{y}) & A \\
\mathcal{C}(\mathfrak{x}) & \mathcal{C}(\mathfrak{g}). \\
};
\path[->]
(m-1-1) edge (m-1-2)
(m-1-1) edge node[left] {$\mathcal{C}(\varphi,\kappa)$} (m-2-1)
(m-1-2) edge node[right] {$p$} (m-2-2)
(m-2-1) edge (m-2-2);
\end{tikzpicture}
\end{center}
There exists a lift $H\colon \mathcal{C}(\mathfrak{x})\to A$ in this diagram because $p$ is an acyclic fibration, $\mathcal{C}(\varphi,\kappa)$ is a cofibration (see Lemma \ref{lemma_swaps_cofibs_and_fibs}), and the category $\mathscr{A}$ is a model category. Thus, the proof is complete.
\end{proof}

\begin{theorem}\label{thm_main}
The category of curved Lie algebras with curved morphisms defines a model category with the model structure defined in Definition \ref{def_model_structure_of_curved_Lie}. Moreover, it is Quillen equivalent to the model category of unital commutative differential graded algebras.
\end{theorem}
\begin{proof}
It follows from the preceding results.
\end{proof}

\section{Unbased rational homotopy theory}\label{sec_rat_homo}

For the remainder of the paper it is assumed that the ground field is the rational numbers, $\mathbb{Q}$. Within this section, results contained in \cite{bousfield_gugenheim} and \cite{laz_markl} will be used alongside the Quillen equivalence of Theorem \ref{thm_main} to construct a disconnected rational homotopy theory using the category of pseudo-compact curved Lie algebras.

The category of simplicial sets will be denoted by $\mathscr{S}$. Recall that, the category $\mathscr{S}$ possesses a well known model structure where the weak equivalences are weak homotopy equivalences. More details can be found, for instance, in \cite{bousfield_gugenheim} and \cite{dwyer_spalinski}. With this model structure all simplicial sets are cofibrant and the fibrant objects are the Kan complexes.

First some definitions are recalled.

\begin{definition}
A group, $G$, is uniquely divisible if the equation $x^r = g$ has a unique solution $x\in  G$ for each $g \in G$ and $r>1$.
\end{definition}

\begin{definition}
A connected Kan complex, $X$, is said to be
\begin{itemize}
\item nilpotent if for any choice of base vertex
\begin{itemize}
\item its fundamental group, $\pi_1 (X)$, is nilpotent and
\item every other homotopy group, $\pi_i (X)$, is acted on nilpotently the fundamental group, in the sense that the sequence $G_{0,i}=\pi_i (X)$, $G_{k+1,i}\colon= \lbrace gn- n \colon n\in N_{k,i}, g\in\pi_1 (X) \rbrace$ terminates;
\end{itemize}
\item rational if each of its homotopy groups are uniquely divisible.
\end{itemize}
\end{definition}

\begin{definition}
A cdga is said to be connected if it is concentrated in non-negative degrees and equal to the ground field, $\mathbb{Q}$, in degree $0$. Similarly, a cdga $X$ is said to be homologically connected if $H^0(X)=\mathbb{Q}$ and $H^i (X)=0$ for all $i<0$. Extending this notion, a cdga $X$ is said to be homologically disconnected cdga if $H^0(X)$ is isomorphic to a finite product of copies of $\mathbb{Q}$ and $H^i(X)=0$ for all $i<0$.
\end{definition}

\begin{remark}
It is worth noting that a homologically disconnected cdga may in fact be homologically connected.
\end{remark}

\begin{definition}
Given some category $X(=\curvedLie_* ,\mathscr{A},\mathscr{S})$, let $\ho(X)$ denote the homotopy category of $X$. Given some category $X(=\mathscr{A},\mathscr{S})$, let $X^c$ denote the full subcategory of connected objects and let $X^{dc}$ denote the full subcategory of objects with finitely many connected components.
\end{definition}

For example, the objects of the category $\ho (\mathscr{S}^c)$ are connected Kan complexes.

\begin{definition}
A cdga $A\in\mathscr{A}$ is called a Sullivan algebra if its underlying graded algebra is the free commutative graded algebra $\wedge V$ on a graded vector space $V$ satisfying the following: $V$ is the union of an increasing series of graded subspaces $V_0 \subset V_1 \subset \dots$, with $d\equiv 0$ on $V_0$ and $d(V_k)$ is contained in $\wedge (V_{k-1})$.

A Sullivan algebra is called minimal if the image of the differential $d$ is contained in $\wedge^+ (V)^2$, where $\wedge^+ (V)$ is the direct sum of the positive degree subspaces of $\wedge (V)$.
\end{definition}

Recall that every non-negatively graded homologically connected cdga admits a minimal model, cf.~\cite[Section 7]{bousfield_gugenheim}. That is, every non-negatively graded homologically connected cdga is weakly equivalent to a minimal algebra. Moreover, the minimal algebra is unique up to isomorphism.

\begin{definition}\hfill
\begin{itemize}
\item A cofibrant homologically connected cdga is said to be of finite type over $\mathbb{Q}$ if its minimal model has finitely many generators in each degree.
\item A nilpotent connected Kan complex is said to be of finite type over $\mathbb{Q}$ if its homology groups with coefficients in $\mathbb{Q}$ are finite dimensional vector spaces over $\mathbb{Q}$.
\end{itemize}
\end{definition}

The adjective `finite type' will be understood to mean `finite type over $\mathbb{Q}$'.

\begin{definition}
The prefix $\mathrm{fN}\mathbb{Q}\mathrm{-}$ applied to a category of simplicial sets will denote the full subcategory composed of rational, nilpotent objects of  finite type over $\mathbb{Q}$, and the prefix $\mathrm{f}\mathbb{Q}\mathrm{-}$ applied to a category of algebras will denote the full subcategory of objects of finite type over $\mathbb{Q}$.
\end{definition}

For example, the category $\mathrm{fN}\mathbb{Q}\mathrm{-}\ho(\mathscr{S}^c)$ denotes the full subcategory of $\ho(\mathscr{S^c})$ of rational, nilpotent Kan complexes of finite type.

Any non-negatively graded homologically disconnected cdga is isomorphic to a finite product of homologically connected cdgas, cf.~\cite[Theorem B]{laz_markl}. Further, the homologically disconnected cdga is said to be of finite type if each connected cdga in the finite product is of finite type, see Proposition 4.4 in op.~cit.

Let $\mathscr{A}_{\geq 0}$ denote the category of non-negatively graded cdgas. Recall, from \cite[Theorem 9.4]{bousfield_gugenheim}, there exists a pair of adjoint functors
\[
F\colon \mathscr{A}_{\geq 0} \leftrightarrows \mathscr{S} \colon  \Omega ,
\]
that induce an equivalence of the homotopy categories $\finA^c_{\geq 0}$ and $\finS^c$. This is the so-called Sullivan-de Rham equivalence. Here $\Omega$ is the de Rham functor \cite[Section 2]{bousfield_gugenheim} and $F$ is the functor given by $X\mapsto F(X,\mathbb{Q})$ where $F(X,\mathbb{Q})$ is the function space of \cite[Section 5.1]{bousfield_gugenheim} (called the Bousfield-Kan functor). Note there is also an analogue of this result in the pointed case contained in op.~cit. Combining the equivalence of these homotopy categories with \cite[Proposition 3.5]{laz_markl}---where the authors extend the existence of minimal models to arbitrary homologically connected cdgas to prove the categories $\ho \mathscr{A}^c_{\geq 0}$ and $\ho\mathscr{A}^c$ are equivalent---it follows that the categories $\finS^c$ and $\finA^c$ are equivalent. Further, one has the following theorem of Lazarev and Markl \cite[Theorem C]{laz_markl}.

\begin{theorem}[Lazarev-Markl]
The three categories $\finS^{dc}$, $\finA^{dc}_{\geq 0}$, and $\finA^{dc}$ are equivalent.
\end{theorem}

It will be the aim of the rest of this section to describe a subcategory $\finL^{dc}$ of $\ho (\curvedLie_* )$ equivalent to the categories $\finS^{dc}$, $\finA^{dc}_{\geq 0}$, and $\finA^{dc}$, extending the theorem of Lazarev and Markl.

\begin{proposition}
For any homologically connected cdga, $A$, the curved Lie algebra $\hatL (X)$ is weakly equivalent to a dgla, i.e.~a curved Lie algebra with zero curvature.
\end{proposition}
\begin{proof}
First note that any minimal algebra has a unique augmentation, thus being endowed with an augmentation means that this minimal algebra corresponds to a dgla under the contravariant functor $\hatL$. Therefore, there exists a filtered quasi-isomorphism $\hatL (A)\to\hatL (M_A)$ since the contravariant functor $\hatL$ preserves weak equivalences, where $M_A$ is a minimal model for $A$.
\end{proof}

\begin{proposition}
For any homologically disconnected cdga, $A$, the curved Lie algebra $\hatL (A)$ is weakly equivalent to a finite coproduct of dglas of the form $\hatL (M)$, for some minimal algebra $M$.
\end{proposition}
\begin{proof}
Since a homologically disconnected cdga, $A$, is a finite product of homologically connected ones up to isomorphism, $A$ is weakly equivalent to a finite product of minimal algebras. Hence $\hatL (A)$ is weakly equivalent to a finite coproduct of dglas of the form $\hatL (M)$.
\end{proof}

With this proposition in mind the following definitions are made.

\begin{definition}\label{def_dc_Lie}
The category $\finL^{c}$ is the full subcategory of $\ho(\curvedLie_*)$ with objects consisting of curved Lie algebras isomorphic to ones of the form $\hatL (M)$, where $M$ is a minimal cdga of finite type.

The category $\finL^{dc}$ is the full subcategory of $\ho(\curvedLie_*)$ with objects consisting of curved Lie algebras isomorphic to finite coproducts of objects in $\finL^{dc}$.
\end{definition}

The next result is now clear.

\begin{theorem}
The category $\finL^{dc}$ is equivalent to the categories $\finS^{dc}$, $\finA^{dc}_{\geq 0}$, and $\finA^{dc}$. \qed
\end{theorem}

To describe the equivalence in a more specific manner, first recall the definition of the Maurer-Cartan simplicial set.

\begin{definition}
The Maurer-Cartan simplicial set is given by the functor
\[
\MC_\bullet\colon \curvedLie_*\to\mathscr{S}
\]
that maps a pseudo-compact curved Lie algebra $\mathfrak{g}$ to the simplicial space of MC elements in $\mathfrak{g}\otimes \Omega (\Delta^\bullet )$, where $\Omega (\Delta^n)$ is the Sullivan-de Rham algebra of polynomial forms on the standard topological cosimplicial simplex considered as a homologically graded cdga.
\end{definition}

For more details on the Maurer-Cartan simplicial set see \cite{hamilton_laz_noncomm_geo,getzler,laz_markl}, for example. Note that the Sullivan-de Rham algebra of polynomial forms on the standard topological cosimplicial simplex must be considered as a homologically graded cdga so that the resulting object when tensored with a homologically curved Lie algebra is again a homologically curved Lie algebra.

\begin{proposition}
The functors $\MC_\bullet\colon \curvedLie_*\to\mathscr{S}$ and $\hatL\Omega\colon \mathscr{S}\to\curvedLie_*$ form an adjoint pair.
\end{proposition}
\begin{proof}
By definition, there exists an isomorphism of simplicial sets $\MC_\bullet (\mathfrak{g})\cong F \C (\mathfrak{g})$. Therefore, the functors are compositions in the following diagram:
\begin{center}
\begin{tikzpicture}
\matrix (m) [matrix of math nodes, column sep=2em, row sep=2em]
{
\curvedLie_* & \mathscr{A} & \mathscr{S} \\
};
\path[->]
(m-1-1) edge[transform canvas={yshift=2}] node[above] {$\C$} (m-1-2)
(m-1-2) edge[transform canvas={yshift=-2}] node[below] {$\hatL$} (m-1-1)
(m-1-2) edge[transform canvas={yshift=2}] node[above] {$F$} (m-1-3)
(m-1-3) edge[transform canvas={yshift=-2}] node[below] {$\Omega$} (m-1-2);
\end{tikzpicture}.
\end{center}
The contravariant functors $\hatL$ and $\C$ are adjoint by Theorem \ref{thm_main}, and the functors $F$ and $\Omega$ are adjoint by \cite{bousfield_gugenheim}, therefore the result follows.
\end{proof}

Further, it follows (from Proposition \ref{prop_adjoint} and \cite{bousfield_gugenheim}) that the composite functors $\MC_\bullet$ and $\hatL\Omega$ induce adjoint functors upon the homotopy categories of $\curvedLie_*$ and $\mathscr{S}$. Restricting to the categories $\finS^{dc}$ and $\finL^{dc}$ these functors induce mutually inverse equivalences, as the following results show.

\begin{proposition}\label{prop_foo}
Given any connected non-negatively graded cdga, $A$, of finite type, the Lie algebra $\hatL (A)$ is weakly equivalent to an object belonging to the category $\finL^{c}$.
\end{proposition}
\begin{proof}
Since $A$ is connected there exists some unique minimal model, $M_A$, that is also of finite type. Therefore, since $A$ is quasi-isomorphic to $M_A$ the Lie algebras $\hatL (A)$ and $\hatL (M_A)$ are weakly equivalent.
\end{proof}

\begin{proposition}\label{prop_bar}
Given any curved Lie algebra $\mathfrak{g}\in\finL^{c}$ the cdga $\C (\mathfrak{g})$ is quasi-isomorphic to a connected non-negatively graded cdga of finite type.
\end{proposition}
\begin{proof}
The Lie algebra will be weakly equivalent to one of the form $\hatL (M)$ for some minimal algebra $M$ of finite type, thus $\C\hatL (M)$ is quasi-isomorphic to $M$ and thus of the right form.
\end{proof}

\begin{theorem}\label{thm_rat_ho}
The functors $\MC_\bullet$ and $\hatL\Omega$ determine mutually inverse equivalences between the categories $\finL^{dc}$ and $\finS^{dc}$.
\end{theorem}
\begin{proof}
First, note that the objects of the category $\finL^{dc}$ are, up to equivalence, coproducts of objects of the category $\finL^{c}$. Likewise the objects of the category $\finA^{dc}$ are, up to equivalence, products of objects of the category $\finA^{c}$. Hence to show there is an equivalence of $\finL^{dc}$ and $\finA^{dc}$ it is sufficient to show there is an equivalence of $\finL^{c}$ and $\finA^c$. By Propositions \ref{prop_foo} and \ref{prop_bar} the contravariant functors $\C$ and $\mathcal{L}$ restrict to an adjunction
\[
\C \colon \finL^{c} \leftrightarrows \finA^{c} \colon  \mathcal{L}.
\]
Since Theorem \ref{thm_main} shows that the contravariant functors $\C$ and $\hatL$ form a Quillen pair it is evident that their restrictions form mutually inverse equivalences of the categories $\finL^{c}$ and $\finA^{c}$. Whence the categories  $\finL^{dc}$ and $\finA^{dc}$ are equivalent. Combining this with \cite[Theorem C]{laz_markl} there exists equivalences of the categories
\[
\finL^{dc}\sim\finA^{dc}\sim\finA^{dc}_{\geq 0} \sim \finS^{dc},
\]
and the proof is complete.
\end{proof}

Since equivalences of categories preserve colimits and limits, an analogue of \cite[Theorem 1.7]{laz_markl} and its corollary can be explained here immediate consequences of Theorem \ref{thm_rat_ho}.

\begin{corollary}
Given $\coprod_{i\in I}\mathfrak{g}_i\in\finL^{dc}$, the simplicial set $\MC_\bullet (\coprod_{i\in I} \mathfrak{g}_i)$ is weakly equivalent to the disjoint union $\bigcup_{i\in I} \MC_\bullet ( \mathfrak{g}_i)$.
\end{corollary}

Let $\mathcal{MC} (-):=\pi_0 (\MC_\bullet (-))$ denote the Maurer-Cartan moduli set, i.e.~the set of Maurer-Cartan elements up to homotopy. This construction can be found for example in \cite{hamilton_laz_noncomm_geo,getzler,laz_markl}. An alternate construction in the general case of pronilpotent dglas given in \cite{schlessinger_stasheff_deformation_rational} is that the Maurer-Cartan moduli set can be described as the Maurer-Cartan set up to gauge equivalence. For more details regarding gauge equivalence and a proof of this statement, see \cite{chuang_laz_feynman,voronov_non-abelian}.

\begin{corollary}
Given $\coprod_{i\in I}\mathfrak{g}_i\in\finL^{dc}$, there exists an isomorphism of sets
\[
\mathcal{MC}\left(\coprod_{i\in I} \mathfrak{g}_i \right)\cong\bigcup_{i\in I}\mathcal{MC}(\mathfrak{g}_i).
\]
\end{corollary}

Theorem \ref{thm_rat_ho} also has an application to mapping spaces. Recall that given a pseudo-compact curved Lie algebra, $(\mathfrak{g}, d_\mathfrak{g} , \omega_\mathfrak{g} )$, and a unital cdga, $A$, (both homologically graded) their completed tensor product possesses a well defined pseudo-compact curved Lie algebra structure, see Proposition \ref{prop_completed_tensor_Lie_structure}. Given such a tensor product, the MC elements can be studied in the usual manner as solutions to the MC equation
\[
\omega_\mathfrak{g} \hat{\otimes} 1 + d\xi +\frac{1}{2}[\xi ,\xi ].
\]
The MC elements of such a tensor product correspond to morphisms of cdgas, as the following shows. Let $\MC(\mathfrak{g},A):=\MC(\mathfrak{g}\hat{\otimes}A)$ be considered as a bifunctor.

\begin{proposition}\label{prop_iso_functors}
Given a pseudo-compact curved Lie algebra, $(\mathfrak{g},d_\mathfrak{g},\omega_\mathfrak{g})$, and a cdga, $A$, the two bifunctors $\MC(\mathfrak{g},A)$ and $\Hom_\mathscr{A} (\mathcal{C}(\mathfrak{g}),A)$ are naturally isomorphic.
\end{proposition}
\begin{proof}
Given some degree minus one element of $\mathfrak{g}\hat{\otimes} A$, it corresponds precisely to a continuous linear morphism $k\to(\Sigma^{-1}\mathfrak{g} )\hat{\otimes}A$ of degree zero. This continuous linear morphism, in turn, defines (and is defined by) a linear morphism $(\Sigma^{-1}\mathfrak{g})^*\to A$ which extends uniquely to a morphism of graded commutative algebras $\mathcal{C}(\mathfrak{g})\to A$. The condition that this morphism is in fact one of cdga is precisely the one that the original element is a MC element.
\end{proof}

This result extends to the level of homotopy. It is necessary to first, however, to recall the definition of a homotopy of MC elements.

\begin{definition}
Let $k[z,dz]$ be the free unital cdga on generators $z$ and $dz$ of degrees $0$ and $1$ respectively, subject to the condition $d(z)=dz$.

Given some unital cdga $A$, let $A[z,dz]$ denote the unital cdga given by the tensor $A\otimes k[z,dz]$. Further, denote the quotient morphisms by setting $z=0,1$ by
\[
|_0,|_1 \colon A[z,dz]\to A.
\]
\end{definition}

\begin{definition}
Given a pseudo-compact curved Lie algebra, $\mathfrak{g}$, and a unital cdga, $A$, two elements $\xi,\eta \in \MC(\mathfrak{g},A)$ are said to be homotopic if there exists
\[
h\in \MC(\mathfrak{g},A[z,dz])
\]
such that $h|_0=\xi$ and $h|_1 =\eta$.
\end{definition}

As Proposition \ref{prop_iso_functors} shows, a homotopy of MC elements is nothing more than a Sullivan homotopy, i.e.~a right homotopy with path object $A[z,dz]$. Therefore, two elements of $\mathcal{MC}(\mathfrak{g},A)$ belong to the same class if, and only if, the two corresponding morphisms $\mathcal{C}(\mathfrak{g})\to A$ belong to the same homotopy class.

\begin{corollary}\label{cor_mapping_space}
Given $X,Y\in\finS^{dc}$, then
\[
\Hom_{\finS^{dc}}(X,Y)\cong \mathcal{MC}(\mathcal{L}\Omega (Y)\otimes \Omega (X)).
\]
\end{corollary}
\begin{proof}
There is clearly an isomorphism
\[
\Hom_{\finS^{dc}}(X,Y) \cong \Hom_{\finA^{dc}} (\Omega (Y),\Omega(X)).
\]
It therefore suffices to show that there exists an isomorphism of $\mathcal{MC}(\mathcal{L}\Omega (Y)\otimes \Omega (X))$ and homotopy classes of morphisms $\Omega (Y) \to \Omega (X)$ which is contained within the discussion above.
\end{proof}

\begin{remark}\label{rem_models}
In \cite[Theorem 8.1]{andrey_MC} an explicit model for every connected component of the mapping space between two connected rational, nilpotent CW complexes of finite type, $X$ and $Y$, is given; it is further assumed that either $X$ is a finite CW complex or $Y$ has a finite Postnikov tower, because this ensures the spaces of maps between $X$ and $Y$ are homotopically equivalent to finite type complexes. Whereas the result Corollary \ref{cor_mapping_space} gives a model for the whole mapping space. Therefore, in the case when the two spaces, $X$ and $Y$, are both sufficiently nice (i.e.~are both composed of finitely many connected components each rational, nilpotent, and of finite type with either $X$ being a finite CW complex or $Y$ having a finite Postnikov tower and such that the space of maps has finitely many connected components), the results \cite[Theorem 8.1]{andrey_MC} and Corollary \ref{cor_mapping_space} could be combined to construct a model for the mapping space as a coproduct of finitely many MC moduli spaces. The material developed in this paper, however, does not allow the extension to the case where at least one of the spaces, $X$ or $Y$, fails to meet the aforementioned constraints, or to the case when the space of maps has infinitely many connected components, and this can happen in some seemingly straightforward cases; the space of maps between $2$-dimensional spheres, for example, has infinitely many connected components.
\end{remark}

\section*{Acknowledgements}{The author would like to thank Andrey Lazarev for many helpful and fruitful discussions, and for his corrections and remarks on the first draft of the paper. Paul Levy and Theodore Voronov are also owed thanks for their helpful comments on the paper.}

\bibliography{my_bib}{}

\def\cprime{$'$}
\begin{thebibliography}{BFM11}

\bibitem[Bar68]{barr}
Michael Barr.
\newblock Harrison homology, {H}ochschild homology and triples.
\newblock {\em J. Algebra}, 8:314--323, 1968.

\bibitem[Ber15]{berglund}
Alexander Berglund.
\newblock Rational homotopy theory of mapping spaces via {L}ie theory for
  {$L_\infty$}-algebras.
\newblock {\em Homology Homotopy Appl.}, 17(2):343--369, 2015.

\bibitem[BFM09]{BFM}
Urtzi Buijs, Yves F\'elix, and Aniceto Murillo.
\newblock Lie models for the components of sections of a nilpotent fibration.
\newblock {\em Trans. Amer. Math. Soc.}, 361(10):5601--5614, 2009.

\bibitem[BFM11]{BFM2}
Urtzi Buijs, Yves F\'elix, and Aniceto Murillo.
\newblock {$L_\infty$} models of based mapping spaces.
\newblock {\em J. Math. Soc. Japan}, 63(2):503--524, 2011.

\bibitem[BG76]{bousfield_gugenheim}
A.~K. Bousfield and V.~K. A.~M. Gugenheim.
\newblock On {${\rm PL}$} de {R}ham theory and rational homotopy type.
\newblock {\em Mem. Amer. Math. Soc.}, 8(179):ix+94, 1976.

\bibitem[BL05]{block_laz}
J.~Block and A.~Lazarev.
\newblock Andr\'e-{Q}uillen cohomology and rational homotopy of function
  spaces.
\newblock {\em Adv. Math.}, 193(1):18--39, 2005.

\bibitem[BPS89]{BPS}
A.~K. Bousfield, C.~Peterson, and L.~Smith.
\newblock The rational homology of function spaces.
\newblock {\em Arch. Math. (Basel)}, 52(3):275--283, 1989.

\bibitem[Bra12]{braun}
Christopher Braun.
\newblock {Operads and moduli spaces}.
\newblock {\em arXiv:1209.1088}, 2012.

\bibitem[BS97]{BES}
Edgar~H. Brown, Jr. and Robert~H. Szczarba.
\newblock On the rational homotopy type of function spaces.
\newblock {\em Trans. Amer. Math. Soc.}, 349(12):4931--4951, 1997.

\bibitem[CL10]{chuang_laz_feynman}
J.~Chuang and A.~Lazarev.
\newblock Feynman diagrams and minimal models for operadic algebras.
\newblock {\em J. Lond. Math. Soc. (2)}, 81(2):317--337, 2010.

\bibitem[CL11]{chuang_laz}
Joseph Chuang and Andrey Lazarev.
\newblock {$L$}-infinity maps and twistings.
\newblock {\em Homology Homotopy Appl.}, 13(2):175--195, 2011.

\bibitem[CLM14]{chuang_laz_mannan}
Joseph Chuang, Andrey Lazarev, and Wajid Mannan.
\newblock {Cocommutative coalgebras: homotopy theory and Koszul duality}.
\newblock {\em arXiv:1403.0774}, 2014.

\bibitem[DS95]{dwyer_spalinski}
William~G. Dwyer and Jan Spali{\'n}ski.
\newblock Homotopy theories and model categories.
\newblock In {\em Handbook of algebraic topology}, pages 73--126.
  North-Holland, Amsterdam, 1995.

\bibitem[Gab62]{gabriel}
Pierre Gabriel.
\newblock Des cat\'egories ab\'eliennes.
\newblock {\em Bull. Soc. Math. France}, 90:323--448, 1962.

\bibitem[Get09]{getzler}
Ezra Getzler.
\newblock Lie theory for nilpotent {$L_\infty$}-algebras.
\newblock {\em Ann. of Math. (2)}, 170(1):271--301, 2009.

\bibitem[Hae82]{haefliger}
Andr\'e Haefliger.
\newblock Rational homotopy of the space of sections of a nilpotent bundle.
\newblock {\em Trans. Amer. Math. Soc.}, 273(2):609--620, 1982.

\bibitem[Har62]{harrison}
D.~K. Harrison.
\newblock Commutative algebras and cohomology.
\newblock {\em Trans. Amer. Math. Soc.}, 104:191--204, 1962.

\bibitem[Hin97]{hinich_homo_alg}
Vladimir Hinich.
\newblock Homological algebra of homotopy algebras.
\newblock {\em Comm. Algebra}, 25(10):3291--3323, 1997.

\bibitem[Hin01]{hinich_stacks}
Vladimir Hinich.
\newblock D{G} coalgebras as formal stacks.
\newblock {\em J. Pure Appl. Algebra}, 162(2-3):209--250, 2001.

\bibitem[HL09]{hamilton_laz_noncomm_geo}
Alastair Hamilton and Andrey Lazarev.
\newblock Cohomology theories for homotopy algebras and noncommutative
  geometry.
\newblock {\em Algebr. Geom. Topol.}, 9(3):1503--1583, 2009.

\bibitem[Jar97]{jardine}
J.~F. Jardine.
\newblock A closed model structure for differential graded algebras.
\newblock In {\em Cyclic cohomology and noncommutative geometry ({W}aterloo,
  {ON}, 1995)}, volume~17 of {\em Fields Inst. Commun.}, pages 55--58. Amer.
  Math. Soc., Providence, RI, 1997.

\bibitem[KY11]{keller_yang}
Bernhard Keller and Dong Yang.
\newblock Derived equivalences from mutations of quivers with potential.
\newblock {\em Adv. Math.}, 226(3):2118--2168, 2011.

\bibitem[Laz13]{andrey_MC}
Andrey Lazarev.
\newblock Maurer-{C}artan moduli and models for function spaces.
\newblock {\em Adv. Math.}, 235:296--320, 2013.

\bibitem[LM15]{laz_markl}
Andrey Lazarev and Martin Markl.
\newblock Disconnected rational homotopy theory.
\newblock {\em Adv. Math.}, 283:303--361, 2015.

\bibitem[ML98]{cat_for_work}
Saunders Mac~Lane.
\newblock {\em Categories for the working mathematician}, volume~5 of {\em
  Graduate Texts in Mathematics}.
\newblock Springer-Verlag, New York, second edition, 1998.

\bibitem[Pos11]{positselski}
Leonid Positselski.
\newblock Two kinds of derived categories, {K}oszul duality, and
  comodule-contramodule correspondence.
\newblock {\em Mem. Amer. Math. Soc.}, 212(996):vi+133, 2011.

\bibitem[Qui69]{quillen}
Daniel Quillen.
\newblock Rational homotopy theory.
\newblock {\em Ann. of Math. (2)}, 90:205--295, 1969.

\bibitem[SS12]{schlessinger_stasheff_deformation_rational}
Michael Schlessinger and James Stasheff.
\newblock {Deformation theory and rational homotopy type}.
\newblock {\em arXiv:1211.1647}, 2012.

\bibitem[Sul77]{sullivan}
Dennis Sullivan.
\newblock Infinitesimal computations in topology.
\newblock {\em Inst. Hautes \'Etudes Sci. Publ. Math.}, (47):269--331 (1978),
  1977.

\bibitem[VdB15]{vandenbergh}
Michel Van~den Bergh.
\newblock Calabi-{Y}au algebras and superpotentials.
\newblock {\em Selecta Math. (N.S.)}, 21(2):555--603, 2015.

\bibitem[Vor12]{voronov_non-abelian}
Theodore~Th. Voronov.
\newblock On a non-abelian {P}oincar\'e lemma.
\newblock {\em Proc. Amer. Math. Soc.}, 140(8):2855--2872, 2012.

\bibitem[Wei94]{weibel}
Charles~A. Weibel.
\newblock {\em An introduction to homological algebra}, volume~38 of {\em
  Cambridge Studies in Advanced Mathematics}.
\newblock Cambridge University Press, Cambridge, 1994.

\end{thebibliography}

\end{document}